\documentclass[a4paper,11pt,DIV12]{scrartcl}

\usepackage[utf8]{inputenc}
\usepackage{amsmath}
\usepackage{amssymb}
\usepackage{amsthm}
\usepackage{graphicx}
\usepackage{mathtools}
\usepackage{siunitx}
\usepackage{afterpage}
\usepackage{hyperref}
\usepackage{pgfplots}
\usetikzlibrary{positioning}

\newcommand{\N}{\mathbb{N}}
\newcommand{\Z}{\mathbb{Z}}
\newcommand{\R}{\mathbb{R}}
\newcommand{\C}{\mathbb{C}}
\newcommand{\Cper}{C_{\textnormal{per}}^\infty}
\newcommand{\fib}{\textnormal{fib}}
\newcommand{\Hfib}{H_\fib}
\newcommand{\Dfib}{D_\fib}
\newcommand{\X}{\Omega \times S^2} 
\newcommand{\wto}{\rightharpoonup}
\newcommand{\psireg}{\psi^\delta_{\alpha, h}}
\DeclareMathOperator*{\argmin}{argmin}
\DeclareMathOperator*{\vspan}{span}
\DeclareMathOperator{\grad}{grad}
\DeclareMathOperator{\sgrad}{Grad}
\DeclareMathOperator{\Laplace}{\Delta}
\DeclareMathOperator{\domain}{\mathcal{D}}
\DeclareMathOperator{\range}{\mathcal{R}}
\DeclareMathOperator{\order}{\mathcal{O}}
\DeclarePairedDelimiter{\norm}{\lVert}{\rVert}
\DeclarePairedDelimiter{\abs}{\lvert}{\rvert}
\DeclarePairedDelimiterX{\inner}[2]{\langle}{\rangle}{#1,\, #2}
\DeclarePairedDelimiterX{\comm}[2]{[}{]}{#1,\, #2} 

\newtheorem{theorem}{Theorem}[section]
\newtheorem{corollary}[theorem]{Corollary}
\newtheorem{lemma}[theorem]{Lemma}
\theoremstyle{definition}
\newtheorem{definition}[theorem]{Definition}

\usepackage{etoolbox}
\makeatletter
\newif{\if@label}
\patchcmd\label@in@display{\@empty}{\@empty\@labeltrue}{}{}
\preto\equation{\@labelfalse}
\preto\endequation{\if@label\else\notag\fi}
\makeatother

\begin{document}

\title{A coherence enhancing penalty for Diffusion MRI:\ regularizing property and discrete approximation}

\author{T. Hohage\thanks{Institute for Numerical and Applied Mathematics, University of Göttingen, Germany} \and C. Rügge\footnotemark[1]}

\maketitle

\begin{abstract}
  Processing of Diffusion MRI data obtained from High Angular Resolution measurements consists of a series of steps, starting with the estimation of an orientation distribution function (ODF), which is then used as input for e.g.\ tractography algorithms. It is important that ODF reconstruction methods yield accurate, coherent ODFs, in particular for low SNR or coarsely sampled data sets. As the diffusion process is modelled independently in each voxel, reconstructions are often carried out for each voxel separately, disregarding the observation that neighboring voxels are often quite similar if they belong to the same fiber structure. There are surprisingly few approaches that make use of this kind of spatial regularity to improve coherence and stability of the reconstruction. In this work, we focus on a variation of a method proposed by Reisert and Kiselev based on the concept of fiber continuity. The method has already been shown to yield good numerical results, but has not yet been analyzed theoretically. Under suitable smoothness assumptions, we apply results on constrained Tikhonov-type regularization with approximate operator to show convergence of reconstructions from discrete, noisy data for linear forward models. Further, we numerically illustrate the performance of the method on phantom and \emph{in-vivo} data.
\end{abstract}

\pagestyle{myheadings}
\thispagestyle{plain}
\markboth{T. HOHAGE AND C. RÜGGE}{A COHERENCE ENHANCING PENALTY FOR DIFFUSION MRI}

\section{Introduction}

Diffusion weighted MRI (DW-MRI) is a non-invasive method to measure the movements of water molecules in biological tissue. Using the fact that water diffusion is mainly directed along nerve fibers, not perpendicular to them, it can be used to resolve the fibrous structure of brain white matter, with a wide range of applications in both medicine and neuro-science. DW measurements acquire a number of full 3-dimensional MRI volumes with varying diffusion sensitizing gradients encoding the diffusion in different directions. From this, one tries to infer information on the diffusive properties of the tissue as parametrized by a suitable physical model of the diffusion process. The most widely-used model in this regard, the tensor model, is based on the assumption of anisotropic Gaussian diffusion parametrized by a diffusion tensor in each voxel. While being quite successful, its main drawback is the inherent inability to resolve more than one diffusion direction per voxel. However, depending on the resolution, as much as a third of all voxels can contain multiple fibers.

More refined approaches often replace the diffusion tensor by an \emph{orientation distribution function} (ODF) measuring either --- depending on the underlying model --- the diffusion probability or the density of fibers per direction for each voxel of the volume. This requires a larger number of diffusion sensitizing gradients than for the tensor model, often termed as High Angular Resolution Diffusion Imaging (HARDI).

Most models describe diffusion in each voxel separately, and for efficiency, reconstructions are usually simply carried out voxel-wise. This approach however neglects spatial coherence, i.e.\ the fact that ODFs in nearby voxels are often similar if they belong to the same fiber bundle. Incorporating spatial coherence into the reconstruction algorithm can potentially improve the results significantly. Spatial regularization was first investigated in~\cite{goh-et-al} where similarity is measured by comparing the entire ODFs in nearby voxels. Smoothing methods that take into account the underlying structure of the domain \(\Omega \times S^2\) have been suggested for example in~\cite{duits-franken}, where linear and non-linear diffusion filters are applied to the reconstructed ODF in a post-processing step, and in~\cite{becker-et-al}, where adaptive smoothing is performed on the HARDI data prior to reconstruction.

In this paper we will study the concept of \emph{fiber continuity} (FC) as presented in in~\cite{reisert-kiselev}, which uses related ideas to regularize the ODF reconstruction. Here similarity between ODFs is \emph{local} also in the orientational part and only compares voxels along fibers instead of isotropically. This way, smoothness information can extend for example from single-fiber voxels into adjacent crossings despite the sudden appearance of perpendicular structures which violate the \emph{global} similarity of the respective ODFs. The approach is based on the assumption that curvature of the fiber bundles is not too large, and that the point \((x, u) \in \X\) belongs to a fibrous structure through \(x\), directed along \(u\). So for sufficiently small step-lengths \(\tau > 0\), the ODF field \(\psi \colon \X \to \R\) assigning to each voxel \(x\) in the spatial domain \(\Omega\) the ODF \(\psi(x, \cdot)\) fulfills
\begin{equation}
  \psi(x + \tau u, u) \simeq \psi(x, u).
\end{equation}
This assumption is employed in a Tikhonov-type regularization scheme, with the anisotropic regularization term
\begin{equation}\label{eq:anisotropic-penalty}
  \int\limits_\Omega \int\limits_{S^2} \abs*{u^T \grad_x \psi(x, u)}^2 \,du \,dx.
\end{equation}

Although this regularization scheme leads to good results, its theoretical properties have not been studied yet. In particular, compactness properties, which are needed to establish convergence of discrete approximations, are not obvious since the forward operator is not smoothing in the spatial variable \(x\). The aim of this work is to analyze the convergence of a spatial regularization method similar to the one above for noisy, discrete data. This will be achieved by studying the properties of a non-standard Sobolev-type space constructed from the anisotropic derivative.

The paper is organized as follows: in \S\ref{sec:rec-scheme}, we describe the diffusion model and the regularization method, and we formally introduce the Sobolev-type space \(\Hfib\) used in this method. In \S\ref{sec:reg-with-disc}, we review some convergence results for Tikhonov regularization with approximate operators, with focus on approximation by finite dimensional operators, i.e.\ discretization. In \S\ref{sec:hfib}, the space \(\Hfib\) will be further investigated, and we show the main result of this paper, a compact embedding theorem that allows us to apply the convergence results from \S\ref{sec:reg-with-disc} to our method. Finally, \S\ref{sec:implementation} and~\ref{sec:numerics} describe some implementation details and show some numerical results, comparing the performance of the method to standard unregularized reconstructions.

\section{Reconstruction scheme}
\label{sec:rec-scheme}

Modelling of the DW signal is a difficult problem in itself since the diffusive properties of white matter can be quite complex. We will use the spherical deconvolution (SD) model that was introduced in~\cite{tournier-et-al-2004}, which is particularly attractive due to its conceptual simplicity and linearity. It has been successfully employed in both phantom and \emph{in-vivo} studies.

The model describes the data as \(S^2\)-convolution of the (fiber) ODF \(\psi \colon \X \to \R\) with a single-fiber response function \(k \colon [-1,\, 1] \to \R\) that is assumed to be identical throughout the volume:
\begin{equation}
  T \psi(x, q) := \int\limits_{S^2} k(q^T u) \, \psi(x, u) \,du, \quad q \in S^2.
\end{equation}
The response function \(k\) can either be estimated from the data in a pre-processing step or fixed a priori by modelling it e.g.\ as a Gaussian.

A problem of this model, apart from suffering from instability due to ill-posedness, is that it can lead to ODFs with negative values, which makes interpretation as a diffusion probability or fiber density impossible. Therefore, one often includes non-negativity constraints \(\psi \geq 0\). In the context of SD, non-negativity constraints have been introduced in~\cite{tournier-et-al-2007}, the resulting method being called constrained SD (CSD). Adding the fiber continuity based spatial regularization strategy introduced in~\cite{reisert-kiselev}, the method can then be written as
\begin{equation}\label{eq:optimization-problem-reisert}
  \argmin_{\psi \geq 0} \Bigl(\norm[\big]{T \psi - S^\delta}^2 + \gamma \int\limits_{\X} \abs[\big]{u^T \grad \psi(x, u)}^2 \,dx \,du \Bigr),
\end{equation}
where \(\gamma > 0\) is a regularization parameter and \(S^\delta\) is the given noisy data with
\begin{equation}
  \norm[\big]{S - S^\delta} \leq \delta
\end{equation}
for \(\delta > 0\) and exact data \(S = T \psi^\dagger\). Note that in the context of discretization, \(\delta\) quantifies both measurement errors and errors caused by interpolating the data from a finite sampling grid to all of \(\Omega \times S^2\). Our approach differs from this slightly: instead of~\eqref{eq:optimization-problem-reisert}, we solve
\begin{equation}\label{eq:optimization-problem}
  \begin{split}
    \argmin_{\psi \geq 0} \Bigl(\norm[\big]{T \psi - S^\delta}^2 + \alpha \norm[\big]{\psi}^2 - \beta \inner[\big]{\psi}{\Laplace_{S^2} \psi} \Bigr. \\
    \Bigl. \qquad + \gamma \int\limits_{\X} \abs[\big]{u^T \grad \psi(x, u)}^2 \,dx \,du \Bigr).
  \end{split}
\end{equation}
This is both due to theoretical reasons to be detailed below, and the numerical experience that the additional angular smoothness introduced by the \(\Laplace_{S^2}\)-term improves stability of the reconstruction and reduces artifacts for curved structures. The occurrence of these artifacts is not surprising since the regularization term~\eqref{eq:anisotropic-penalty} was derived for locally straight fibers. However, angular smoothness also tends to blur structures and limit the achievable angular resolution of the ODFs. Therefore, the parameter \(\beta\) has to be chosen carefully.

For a rigorous treatment of this optimization problem, we first need to define the function space in which the problem is posed. The spatial domain \(\Omega\) will be simply taken as a cube in \(\R^3\),
\begin{equation}
  \Omega = \prod_{i=1}^3 \biggl( -\frac{L_i}{2},\, \frac{L_i}{2} \biggr),
\end{equation}
with \(L \in \R_+^3\). In \(L^2(\X)\), introduce the orthonormal basis functions
\begin{equation}\label{eq:basis-definition}
  \eta_{klm} := \xi_k \otimes Y_{lm},
\end{equation}
where \(\xi_k \colon \Omega \to \C\) for \(k \in \Z^3\) is the trigonometric orthonormal basis on \(\Omega\) and \(Y_{lm} \colon S^2 \to \C\) for \(l \in \N_0\) and \(m = -l, \dots, l\) are the Spherical Harmonics on \(S^2\). Let \(\Cper(\X)\) be the space of infinitely differentiable functions on \(\R^3 \times S^2\) the spatial part of which is periodic with period \(L\). For \(\psi \in \Cper(\X)\), we denote by \(\grad \psi\) and \(\sgrad \psi\) the derivatives with respect to the spatial and orientational parts, respectively. Moreover, we extend this notation to functions \(\psi \in L^2(\X)\) using the basis \(\eta_{klm}\), i.e.
\begin{equation}
  \grad \psi = \sum_{klm} \inner{\psi}{\eta_{klm}} (\grad \xi_k) \otimes Y_{lm}
\end{equation}
and
\begin{equation}
  \sgrad \psi = \sum_{klm} \inner{\psi}{\eta_{klm}} \xi_k \otimes (\sgrad Y_{lm}),
\end{equation}
if the series converge in \({L^2(\X)}^3\). Here, we view \(\sgrad Y_{lm}\) as a map \(S^2 \to \C^3\) with \(u^T \sgrad Y_{lm}(u) = 0\) for all \(u \in S^2\). For notational simplicity, the function \(\X \to S^2\), \((x,\, u) \mapsto u\) will be simply denoted by \(u\) in the following.
\begin{definition}
  The \textbf{fiber derivative} of \(\psi \in L^2(\X)\) is defined as
  \begin{equation}
    \Dfib \psi := \sum_{klm} \inner{\psi}{\eta_{klm}} \sum_{i=1}^3 \frac{\partial \xi_k}{\partial x_i} \otimes (u_i \, Y_{lm}),
  \end{equation}
  if the series converges in \(L^2(\X)\). The \textbf{fiber space} \(\Hfib(\X)\) is defined as the space of all \(\psi \in L^2(\X)\) for which
  \begin{equation}\label{eq:def-fib-norm}
    \norm{\psi}_\fib := {\bigl( \norm{\psi}^2 + \norm{\Dfib \psi}^2
    + \norm{\sgrad \psi}^2 \bigr)}^{\frac{1}{2}}
  \end{equation}
  is finite.
\end{definition}

If \(\grad \psi\) exists, then \(\Dfib \psi = u^T \grad \psi\) is just the operator in~\eqref{eq:optimization-problem-reisert}. Note that we could have introduced scaling factors into the definition of \(\norm{\cdot}_\fib\) as in~\eqref{eq:optimization-problem}, but for the theoretical analysis it is sufficient to omit them. The convolution operator \(T\) is now viewed as a map \(T \colon \Hfib(\X) \to L^2(\X)\). The properties of this space will be studied in \S\ref{sec:hfib}.

\section{Regularization with discretization}
\label{sec:reg-with-disc}

In this section, we briefly review some results on regularization with discretization and prove a variant of a result in~\cite{lu-flemming}. Discretization of~\eqref{eq:optimization-problem} is performed as a projection method by introducing orthogonal projections \(P_h\) in \(\Hfib(\X)\) and \(Q_h\) in \(L^2(\X)\), with some discretization parameter \(h > 0\), and replacing \(T\) with
\begin{equation}
  T_h := Q_h T P_h,
\end{equation}
i.e.
\begin{equation}\label{eq:discrete-optimization-problem}
  \psireg = \argmin \bigl\{\norm{T_h \psi - S^\delta}^2 + \alpha \norm{\psi}_\fib^2 \colon \psi \in \Hfib(\X),\, \psi \geq 0 \bigr\}.
\end{equation}
There is a number of papers on regularization with discretization in the unconstrained case, we only cite~\cite{plato-vainikko}. Under a Hölder-type source condition,
\begin{equation}\label{eq:spectral-source-condition}
  \psi^\dagger \in \range\bigl( {(T^* T)}^\mu \bigr)
\end{equation}
and with a suitable parameter choice, one obtains the convergence rate
\begin{equation}\label{eq:plato-vainikko-rate}
  \norm{\psireg - \psi^\dagger} = \order\bigl(\delta^{\frac{2 \mu}{2 \mu + 1}} + \norm{T - T_h}^{2 \mu} \bigr)
\end{equation}
for \(0 < \mu \leq \frac{1}{2}\). Unfortunately, the proof relies on spectral theory and therefore does not generalize to the constrained case. Moreover, the convergence \(T_h \to T\) requires \(T\) to be compact, which is not immediately clear in our case: at least viewed as an operator from \(L^2(\X)\) to itself, the operator \(T\) is not compact due to the spatial part.

To address the first problem, spectral theory can be replaced by variational techniques in regularization theory developed in the last decade. An important part in this approach consists of replacing the spectral source condition~\eqref{eq:spectral-source-condition} by a condition in the form of a variational inequality, a so-called variational source condition. We will assume that
\begin{equation}\label{eq:vsa}
  \beta \norm[\big]{\psi - \psi^\dagger}^2 \leq \norm[\big]{\psi}^2 - \norm[\big]{\psi^\dagger}^2 + \phi\bigl( \norm{T(\psi - \psi^\dagger)} \bigr) \quad \text{for all \(\psi \in C\)},
\end{equation}
some \(\beta > 0\), and a strictly increasing, concave function \(\phi \colon \R_+ \to \R_+\) with \(\phi(0) = 0\). Here, \(C\) is the convex constraint set.

It has been shown in~\cite{hofmann-yamamoto} that the spectral Hölder condition~\eqref{eq:spectral-source-condition} implies the variational source condition~\eqref{eq:vsa} with \(\phi(t) = t^{\frac{2\mu}{2\mu+1}}\) for any set \(C\), i.e.\ even if \(C\) is the whole space. For closed, convex sets \(C\) and \(\phi(t) \sim \sqrt{t}\),~\eqref{eq:vsa} is equivalent (c.f.~\cite{flemming-hofmann}) to the projected source condition
\begin{equation}
  \psi^\dagger \in \range(P_C T^*)
\end{equation}
that is well-known (c.f.~\cite{engl-et-al}) for this type of problem.

In~\cite{lu-flemming}, convergence was shown for non-linear Tikhonov regularization with approximate operators employing a variational smoothness assumption. The following result is a small modification --- with almost identical proof --- of~\cite[Theorem~3.1]{lu-flemming}, specializing it to the case of linear operators but extending it to unbounded constraint sets \(C\), for which the assumption \(\sup_{\psi \in C} \norm{T(\psi) - T_h(\psi)} \to 0\) in~\cite{lu-flemming} is typically not satisfied.
\begin{theorem}\label{thm:convergence-lu-flemming}
  Let \(X\), \(Y\) be Hilbert spaces, \(C \subset X\) closed and convex, \(T,\, T_h \colon X \to Y\) linear operators with \(\norm{T - T_h} \leq \eta_h\), and assume that the variational smoothness assumption~\eqref{eq:vsa} hold true. Given \(S^\delta \in Y\) with \(\norm{S^\delta - T \psi^\dagger} \leq \delta\), let \(\psireg \in C\) be the (unique) minimizer of
  \begin{equation}
    C \ni \psi \mapsto \norm{T_h \psi - S^\delta}^2 + \alpha \norm{\psi}^2.
  \end{equation}
  If \(\alpha > \eta_h^2\), then
  \begin{equation}
    \beta \norm[\big]{\psireg - \psi^\dagger}^2 \leq \frac{6 \eta_h^2 \norm{\psi^\dagger}^2 + 4 \delta^2}{\alpha - \eta_h^2} + {(-\phi)}^*\biggl( -\frac{1}{4(\alpha - \eta_h^2)} \biggr),
  \end{equation}
  where \({(-\phi)}^*(s) := \sup_{t \in \R_+} (s t + \phi(t))\) denotes the Fenchel conjugate of \(-\phi\).
\end{theorem}
\begin{proof}
  Using the triangle inequality, we have
  \begin{equation}
    \begin{split}
      {}& \frac{1}{4} \norm[\big]{T (\psireg - \psi^\dagger)}^2 \\
      \leq {}& \norm[\big]{(T - T_h) \psireg}^2 + \norm[\big]{T_h \psireg - S^\delta}^2 + \norm[\big]{T_h \psi^\dagger - S^\delta}^2 + \norm[\big]{(T - T_h) \psi^\dagger}^2 \\
      \leq {}& \eta_h^2 \bigl( \norm{\psireg}^2 + \norm{\psi^\dagger}^2 \bigr) +\norm[\big]{T_h \psireg - S^\delta}^2 + \norm[\big]{T_h \psi^\dagger - S^\delta}^2.
    \end{split}
  \end{equation}
  Due to this and the minimization property of \(\psireg\),
  \begin{equation}
    \begin{split}
      {}& \alpha \norm[\big]{\psi}^2 - \alpha \norm[\big]{\psi^\dagger}^2 \leq \norm[\big]{T_h \psi^\dagger - S^\delta}^2 - \norm[\big]{T_h \psireg - S^\delta}^2 \\
      \leq {}& 2 \norm[\big]{T_h \psi^\dagger - S^\delta}^2 + \eta_h (\norm{\psi^\dagger}^2 + \norm{\psireg}^2) - \frac{1}{4} \norm[\big]{T (\psireg - \psi^\dagger)}^2 \\
      \leq {}& 2 {(\eta_h \norm{\psi^\dagger} + \delta)}^2 + \eta_h^2 (\norm{\psi^\dagger}^2 + \norm{\psireg}^2) - \frac{1}{4} \norm[\big]{T (\psireg - \psi^\dagger)}^2 \\
      \leq {}& 5 \eta_h^2 \norm[\big]{\psi^\dagger}^2 + 4 \delta^2 + \eta_h^2 \norm[\big]{\psireg}^2 - \frac{1}{4} \norm[\big]{T (\psireg - \psi^\dagger)}^2.
    \end{split}
  \end{equation}
  Let \(\gamma > 0\) be arbitrary. Inserting the above inequality into assumption~\eqref{eq:vsa} yields
  \begin{equation}
    \begin{split}
      {}& \beta \norm[\big]{\psireg - \psi^\dagger} \\
      \leq {}& \gamma (\norm{\psi^\dagger}^2 - \norm{\psireg}^2) + (1+\gamma) (\norm{\psireg}^2 - \norm{\psi^\dagger}^2) + \phi\bigl( \norm{T (\psireg - \psi^\dagger)}^2 \bigr) \\
      \leq {}& \gamma \norm[\big]{\psi^\dagger}^2 + \frac{1+\gamma}{\alpha} (5 \eta_h^2 \norm{\psi^\dagger}^2 + 4 \delta^2) + \biggl( \frac{(1+\gamma) \eta_h^2}{\alpha} - \gamma \biggr) \norm[\big]{\psireg}^2 \\
      {}& \qquad -\frac{1+\gamma}{4 \alpha} \norm[\big]{T (\psireg - \psi^\dagger)}^2 + \phi\bigl( \norm{T (\psireg - \psi^\dagger)}^2 \bigr) \\
      \leq {}& \gamma \norm[\big]{\psi^\dagger}^2 + \frac{(1+\gamma)}{\alpha} (5 \eta_h^2 \norm{\psi^\dagger}^2 + 4 \delta^2) + \biggl( \frac{(1+\gamma) \eta_h^2}{\alpha} - \gamma \biggr) \norm[\big]{\psireg}^2 \\
      {}& \qquad + {(-\phi)}^*\Bigl( -\frac{1+\gamma}{4 \alpha} \Bigr) \\
    \end{split}
  \end{equation}
  If we choose \(\gamma\) such that
  \begin{equation}
    \frac{(1+\gamma) \eta_h^2}{\alpha} - \gamma = 0,
  \end{equation}
  which requires \(\alpha > \eta_h^2\) in order for \(\gamma\) to be positive, we finally obtain the desired estimate.
\end{proof}
Using an a-priori parameter choice similar to~\cite{flemming}, which for differentiable \(\phi\) amounts to
\begin{equation}
  \alpha(\delta, h) \sim \eta_h^2 + \frac{1}{\phi'(\eta_h^2 \norm{\psi^\dagger}^2 + \delta^2)} > \eta_h^2,
\end{equation}
leads to a convergence rate
\begin{equation}
  \norm[\big]{\psireg - \psi^\dagger}^2 = \order\bigl( \phi(\eta_h^2 \norm{\psi^\dagger}^2 + \delta^2) \bigr).
\end{equation}
For \(C = X\) and \(\phi(t) \sim \sqrt{t}\), which corresponds to \(\mu = \frac{1}{2}\) in~\eqref{eq:spectral-source-condition}, this rate is identical to~\eqref{eq:plato-vainikko-rate} with respect to the data error \(\delta\), but somewhat worse with respect to the operator error \(\eta_h\).

\section{Properties of the space \(\Hfib(\X)\)}
\label{sec:hfib}

What remains to analyze is the convergence \(T_h \to T\). As mentioned before, this would not be possible if \(T\) was defined on \(L^2(\X)\). In the following, we will show that \(\Hfib(\X)\) is compactly embedded in \(L^2(\X)\), thus enabling convergence if the projections converge pointwise.

We introduce the operator \(D_0 \colon \Hfib(\X) \subset L^2(\X) \to {L^2(\X)}^3 \times {L^2(\X)}^3\) by
\begin{equation}
  D_0 := \begin{pmatrix} \Dfib \\ \sgrad \end{pmatrix}
\end{equation}
in order to write
\begin{equation}
  \norm[\big]{\psi}_\fib^2 = \norm[\big]{\psi}^2 + \norm[\big]{D_0 \psi}^2.
\end{equation}

Moreover, we will need Sobolev spaces on \(\X\). For this purpose, define the operator \(\Lambda^s\) for \(s \in \R\) by
\begin{equation}
  \Lambda^s \psi := \sum_{klm} \lambda_{kl}^s \inner{\psi}{\eta_{klm}} \eta_{klm}, \qquad \lambda_{kl} = \sqrt{1 + \norm{k}^2 + l^2},
\end{equation}
whenever the series converges. Then the Sobolev space \(H^s(\X)\) of order \(s \geq 0\) can be written as the set of all \(\psi \in L^2(\X)\) for which \(\Lambda^s \psi \in L^2(\X)\), equipped with the inner product
\begin{equation}
  \inner{\psi}{\phi}_{H^s(\X)} = \inner{\Lambda^s \psi}{\Lambda^s \phi}.
\end{equation}
It is easy to see that a function \(\psi \in L^2(\X)\) belongs to \(H^1(\X)\) if and only if both \(\grad \psi\) and \(\sgrad \psi\) exist in \(L^2(\X)\), and that there is an equivalence of norms
\begin{equation}
  \norm{\cdot}^2_{H^1(\X)} \simeq \norm{\cdot}^2 + \norm{\grad \cdot}^2 + \norm{\sgrad \cdot}^2.
\end{equation}

We first need some auxiliary results.

\begin{lemma}
  \(\Hfib(\X)\) equipped with the norm \(\norm{\cdot}_\fib\) is a Hilbert space. It coincides with the domain of definition of the self-adjoint operator \(\abs{D_0} := {(D_0^* D_0)}^{1/2}\).
\end{lemma}
\begin{proof}
  \(\Dfib\) is closed, as follows from the weak characterization
  \begin{equation}
    \Dfib \psi = \phi \iff \inner{\phi}{g} = -\inner{\psi}{\Dfib g} \quad \text{for all \(g \in \Cper(\X)\)}.
  \end{equation}
  This and the closedness of \(\sgrad\) imply that \(D_0\) is closed as well. Therefore, its domain \(\domain(D_0) = \Hfib(\X)\) is a Hilbert space with the graph norm \(\norm{\cdot}_{\fib} = {(\norm{\cdot}^2 + \norm{D_0 \cdot}^2)}^{\frac{1}{2}}\).

  Self-adjointness of \(D_0^* D_0\) (initially defined for example on \(C^\infty(\X)\)) can be established by a Friedrich's extension (see e.g.~\cite[Appendix A.8]{taylor}). Hence \(\abs{D_0}\) is well defined by the functional calculus. By the polar decomposition of closed operators in Hilbert spaces we have \(\domain(\abs{D_0}) = \domain(D_0)\).
\end{proof}

\begin{lemma}\label{lemma:c-inf-dense-in-hfib}
  \(\Cper(\X)\) and \(\vspan\{\eta_{klm}\}\) are dense in \(\Hfib(\X)\).
\end{lemma}
\begin{proof}
  Since \(\vspan\{\eta_{klm}\} \subset \Cper(\X)\), it suffices to consider \(\vspan\{\eta_{klm}\}\).

  In a first step, we show that \(\vspan\{\eta_{klm}\}\) is dense in \(\domain(\abs{D_0}^2) \subset \Hfib(\X)\). For \(n \in \N\), let the orthogonal projections \(\tau_n \colon L^2(\X) \to L^2(\X)\) be given by
  \begin{equation}
    \tau_n \psi := \sum_{klm}^n \inner{\psi}{\eta_{klm}}_{L^2} \eta_{klm},
  \end{equation}
  where the sum runs over all \((k,l,m)\) with \(\abs{k} \leq n\), \(\abs{l} \leq n\) and \(\abs{m} \leq l\). Then \(\tau_n \psi \to \psi\) in \(L^2(\X)\) as \(n \to \infty\) since the orthonormal system \(\{\eta_{klm}\}\) is complete. It suffices to show that
  \begin{equation}\label{eq:density}
    \lim_{n \to \infty}\, \norm{\tau_n \psi - \psi}_\fib = 0 \quad \text{for all \(\psi \in \domain(\abs{D_0}^2) \subset \Hfib(\X)\)}.
  \end{equation}
  For all \(\phi \in \Cper(\X)\),
  \begin{equation}
    \inner[\big]{\abs{D_0} \tau_n \psi}{\phi} \to \inner[\big]{\abs{D_0} \psi}{\phi},
  \end{equation}
  and since \(\Cper(\X)\) is dense in \(L^2(\X)\), it follows that \(\abs{D_0} \tau_n \psi \wto \abs{D_0} \psi\). Moreover, \(\abs{D_0}^2 \tau_n \phi \to \abs{D_0}^2 \phi\). This can be seen using \(\abs{D_0}^2 = -\Laplace_{S^2} - \Dfib^2\),
  \begin{equation}
    \Laplace_{S^2} \tau_n \phi = \tau_n \Laplace_{S^2} \phi \to \Laplace_{S^2} \phi
  \end{equation}
  and
  \begin{equation}
    \Dfib^2 \tau_n \phi = \sum_{i,j=1}^3 u_i u_j (\tau_n \partial_{x_i} \partial_{x_j} \phi) \to \sum_{i=1}^3 u_i u_j \partial_{x_i} \partial_{x_j} \phi = \Dfib^2 \phi,
  \end{equation}
  where we have used that both spatial and angular derivatives commute with
  \(\tau_n\). Therefore,
  \begin{equation}
    \inner[\big]{\tau_n \abs{D_0}^2 \tau_n \psi}{\phi} = \inner[\big]{\tau_n \psi}{\abs{D_0}^2 \tau_n \phi} \to \inner[\big]{\psi}{\abs{D_0}^2 \phi} = \inner[\big]{\abs{D_0}^2 \psi}{\phi},
  \end{equation}
  i.e.\ \(\tau_n \abs{D_0}^2 \tau_n \psi \wto \abs{D_0}^2 \psi\), and hence \(\norm{\abs{D_0} \tau_n \psi} \to \norm{\abs{D_0} \psi}\). Since weak
  convergence and convergence of the norms implies strong convergence, we obtain~\eqref{eq:density}.

  The general case can be reduced to the special case above by density of \(\domain(\abs{D_0}^2)\) in \(\Hfib(\X) = \domain(\abs{D_0})\) shown below: For \(\psi \in \domain(\abs{D_0})\) and \(\epsilon > 0\) there exists \(\psi_\epsilon \in \domain(\abs{D_0}^2)\) such that \(\norm{\psi - \psi_\epsilon}_\fib \leq \epsilon/2\) and \(n \in \N\) such that \(\norm{\psi_\epsilon - \tau_n \psi_\epsilon} \leq \epsilon/2\). Hence, by the triangle inequality \(\norm{\psi-\tau_n\psi_\epsilon} \leq \epsilon\).

  The density of \(\domain(\abs{D_0}^2)\) in \(\Hfib(\X) = \domain(\abs{D_0})\) can be shown by considering the spectral decomposition \({(E_\lambda)}_{\lambda \in \R}\) of \(\abs{D_0}\). For \(\psi \in \domain(\abs{D_0})\), let
  \begin{equation}
    \psi_n = E_n \psi - E_{-n} \psi, \qquad n \in \N.
  \end{equation}
  Then
  \begin{equation}
    \norm[\big]{\abs{D_0}^2 \psi_n} = \int\limits_{-n-}^{n+} \lambda^4 \,d \norm[\big]{E_\lambda \psi}^2 \leq n^4 \norm[\big]{\psi}^2,
  \end{equation}
  i.e.\ \(\psi_n \in \domain(\abs{D_0}^2)\). Moreover, \(\psi_n \to \psi\) and \(\abs{D_0} \psi_n = E_n \abs{D_0} \psi - E_{-n} \abs{D_0} \psi \to \abs{D_0} \psi\).
\end{proof}

Additionally, some regularity properties of the operators introduced above and their commutators are needed. These could be obtained easily using the formalism of pseudo-differential operators. However, for the special cases considered here, they can also be proved directly.
\begin{lemma}\label{lemma:regularity-of-commutators}
  For each \(s \in \R\), the operators
  \begin{itemize}
  \item \(\grad,\, \sgrad,\, \Dfib \colon H^s(\X) \to H^{s-1}(\X)\) and their adjoints,
  \item \(\comm{\Lambda^{-1}}{u_i} \colon H^s(\X) \to H^{s+2}(\X)\),
  \item and \(\comm{\Lambda^{-1}}{\grad},\, \comm{\Lambda^{-1}}{\sgrad_i} \colon H^s(\X) \to H^{s+1}(\X)\)
  \end{itemize}
  are bounded. Here, \(\comm{A}{B} := AB - BA\) denotes the commutator.
\end{lemma}
\begin{proof}
  The first statement is obvious from the definitions. For the second one,
  \begin{equation}
    \comm{\Lambda^{-1}}{u_i} \psi = \sum_k \sum_{lm} \sum_{l'm'} (\lambda_{kl}^{-1} - \lambda_{kl'}^{-1}) \inner{\psi}{\eta_{kl'm'}} \inner{u_i Y_{l'm'}}{Y_{lm}} \eta_{klm}.
  \end{equation}
  Since \(u_i\) can be written as a linear combination of spherical harmonics of degree 1, \(\inner{u_i Y_{lm}}{Y_{l'm'}}\) vanishes if \(\abs{l - l'} > 1\). Using this, we obtain that
  \begin{equation}
    \abs{\lambda_{kl}^{-1} - \lambda_{kl'}^{-1}} = \order\bigl( l \, \lambda_{kl}^{-3} \bigr) \qquad \text{as \(k, l \to \infty\)}.
  \end{equation}
  Together with the boundedness of \(\inner{u_i Y_{l m}}{Y_{l'm'}}\), the assertion follows. Similarly,
  \begin{equation}
    \comm{\Lambda^{-1}}{\sgrad_i} \psi
    = \sum_k \sum_{lm} \sum_{l'm'} (\lambda_{kl}^{-1} - \lambda_{kl'}^{-1})
    \inner{\psi}{\eta_{kl'm'}} \inner{\sgrad_i Y_{l'm'}}{Y_{lm}} \eta_{klm}.
  \end{equation}
  Again, \(\inner{\sgrad_i Y_{l'm'}}{Y_{lm}}\) vanishes for \(\abs{l - l'} > 1\) and behaves like \(\order(l)\) for \(l \to \infty\), as can be seen from
  \begin{equation}
    \begin{aligned}
      \sum_i \abs[\big]{\inner{\sgrad_i Y_{l'm'}}{Y_{lm}}}^2 &\leq \sum_i \norm[\big]{\sgrad_i Y_{l'm'}}^2 \norm[\big]{Y_{lm}}^2 = \sum_i \norm[\big]{\sgrad_i Y_{l'm'}}^2 \\
      &= \norm[\big]{\sgrad Y_{l'm'}}^2 = l'(l'+1).
    \end{aligned}
  \end{equation}
  Together with \(\comm{\Lambda^{-1}}{\grad} = 0\), this shows the last assertion.
\end{proof}

We can now prove the embedding theorem. The proof uses techniques related to the proof of Hörmanders theorem on hypo-elliptic operators (c.f.~\cite{hormander}) found in~\cite{helffer-nier}.
\begin{theorem}
  \(\Hfib(\X)\) is a continuously embedded subspace of \(H^{1/2}(\X)\).
\end{theorem}
\begin{proof}
  Assume first that \(\psi \in \Cper(\X)\). Let \(D_1 := \comm{\sgrad}{\Dfib}\). Then, since \(\sgrad (a^T \cdot)(u) = (1 - u u^T) a\) for \(a \in \R^3\), we have
  \begin{equation}
    \begin{split}
      D_1 \psi &= \sum_{i=1}^3 (\sgrad u_i) \grad_i \psi \\
      &= (1 - u u^T) \grad \psi.
    \end{split}
  \end{equation}
  This implies
  \begin{equation}
    \norm[\big]{\Dfib \psi}^2 + \norm[\big]{D_1 \psi}^2 = \norm[\big]{\grad \psi}^2
  \end{equation}
  and hence
  \begin{equation}
    \Lambda^2 = 1 + D_0^* D_0 + D_1^* D_1.
  \end{equation}
  Thus
  \begin{equation}\label{eq:estimate-1}
    \begin{split}
      {}& \norm[\big]{\Lambda^{1/2} \psi}^2 = \inner[\big]{\psi}{\Lambda^{-1} \Lambda^2 \psi} = \norm[\big]{\Lambda^{-1/2} \psi}^2 + \sum_{i = 0}^1 \inner[\big]{\psi}{\Lambda^{-1} D_i^* D_i \psi} \\
      = {}& \norm[\big]{\Lambda^{-1/2} \psi}^2 + \sum_{i = 0}^1 \inner[\big]{D_i \psi}{\Lambda^{-1} D_i \psi} + \inner[\big]{\psi}{ \comm{\Lambda^{-1}}{D_i^{*}} \Lambda^{1/2} \Lambda^{-1/2} D_i \psi } \\
      \leq {}& \norm[\big]{\psi}^2 + \sum_{i = 0}^1 \norm[\big]{\Lambda^{-1/2} D_i \psi}^2 + c \norm[\big]{\psi} \norm[\big]{\Lambda^{-1/2} D_i \psi} \\
      \leq {}& c {\bigl( \norm{\psi} + \norm{D_0 \psi} + \norm{\Lambda^{-1/2} D_1 \psi} \bigr)}^2, \\
    \end{split}
  \end{equation}
  where all generic constants are denoted by \(c\) for simplicity. We used that \(\Lambda^{-1/2}\) and \(\comm{\Lambda^{-1}}{D_i^*} \Lambda^{1/2}\) (\(i=1, 2\)) are bounded. The latter follows using the regularity properties in Lemma~\ref{lemma:regularity-of-commutators}, the Leibniz rule for commutators, i.e.\ \(\comm{A}{B C} = \comm{A}{B} C + B \comm{A}{C}\), and the expression
  \begin{equation}\label{eq:adjoint-spherical-grad-components}
    \sgrad_i^* \psi = -\sgrad_i \psi + 2 u_i \psi,
  \end{equation}
  for the (formal) adjoint of the \(i\)-th component of \(\sgrad\). This in turn can be deduced using partial integration and the relation
  \begin{equation}
    \begin{gathered}
      \int\limits_{S^2} \sgrad_i f(u) \,du = \int\limits_{S^2} {(\sgrad u_i)}^T \sgrad f(u) \,du \\
      = -\int\limits_{S^2} (\Laplace_{S^2} u_i) f(u) \,du = 2 \int\limits_{S^2} u_i f(u) \,du,
    \end{gathered}
  \end{equation}
  where \(\Laplace_{S^2} = -\sgrad^* \circ \sgrad\) and \(\Laplace_{S^2} u_i = -2 u_i\) was used. For the \(D_1\)-term in~\eqref{eq:estimate-1}, we have
  \begin{equation}
    \begin{split}
      {}& \norm[\big]{\Lambda^{-1/2} D_1 \psi}^2 = \sum_{i=1}^3 \inner[\big]{\Lambda^{-1} D_{1,i} \psi}{D_{1,i} \psi} \\
      \leq {}& \sum_{i=1}^3 \abs[\big]{ \inner{\Lambda^{-1} D_{1,i} \psi}{\Dfib \sgrad_i \psi} } + \abs[\big]{ \inner{\Lambda^{-1} D_{1,i} \psi}{\sgrad_i \Dfib \psi} }.
    \end{split}
  \end{equation}
  The first term of this can be estimated by
  \begin{equation}
    \begin{aligned}
      {}& \abs[\big]{ \inner{\Lambda^{-1} D_{1,i} \psi}{\Dfib \sgrad_i \psi} } \\
      \leq {}& \abs[\big]{ \inner{\Lambda^{-1} D_{1,i} \Dfib^* \psi}{\sgrad_i \psi} } + \abs[\big]{ \inner{ \comm{\Dfib^*}{\Lambda^{-1} D_{1,i}} \psi }{\sgrad_i \psi} } \\
      \leq {}& c \bigl( \norm{\Dfib^* \psi}\norm{\sgrad_i \psi} + \norm{\psi} \norm{\sgrad_i \psi} \bigr) \\
      \leq {}& c {\bigl( \norm{\psi} + \norm{D_0 \psi} \bigr)}^2,
    \end{aligned}
  \end{equation}
  where where \(\Dfib^* = -\Dfib\) and the boundedness of \(\Lambda^{-1} D_{1,i}\) and \(\comm{\Dfib^*}{\Lambda^{-1} D_{1,i}}\) was used. Similarly,
  \begin{equation}
    \begin{aligned}
      {}& \abs[\big]{ \inner{\Lambda^{-1} D_{1,i} \psi}{\sgrad_i \Dfib \psi} } \\
      \leq {}& \abs[\big]{ \inner{\Lambda^{-1} D_{1,i} \sgrad_i^* \psi}{\Dfib \psi} } + \abs[\big]{ \inner{\comm{\sgrad_i^*}{\Lambda^{-1} D_{1,i}} \psi}{\Dfib \psi} } \\
      \leq {}& c \bigl( \norm{\sgrad_i^* \psi}\norm{\Dfib \psi} + \norm{\psi} \norm{\Dfib \psi} \bigr),
    \end{aligned}
  \end{equation}
  since \(\comm{\sgrad_i^*}{\Lambda^{-1} D_{1,i}}\) is bounded. Using~\eqref{eq:adjoint-spherical-grad-components} again, the commutator between \(\sgrad_i\) and its adjoint is
  \begin{equation}
    \comm{\sgrad_i}{\sgrad_i^*} \psi = 2 (1 - u_i^2) \psi.
  \end{equation}
  Hence,
  \begin{equation}
    \abs[\big]{ \norm{\sgrad_i^* \psi}^2 - \norm{\sgrad_i \psi}^2 } = \abs[\big]{ \inner{\psi}{\comm{\sgrad_i^*}{\sgrad_i} \psi} } \leq 2 \norm[\big]{\psi}^2,
  \end{equation}
  so the second term can be estimated further by
  \begin{equation}
    \abs[\big]{ \inner{\Lambda^{-1} D_{1,i} \psi}{\sgrad_i \Dfib \psi} } \leq c {\bigl( \norm{\psi} + \norm{D_0 \psi} \bigr)}^2.
  \end{equation}
  Putting everything together, we obtain
  \begin{equation}
    \norm[\big]{\psi}_{H^{1/2}(\X)} = \norm[\big]{\Lambda^{1/2} \psi} \leq c \norm[\big]{\psi}_\fib
  \end{equation}
  for \(\psi \in \Cper(\X)\). Finally, Lemma~\ref{lemma:c-inf-dense-in-hfib} implies that the estimate also holds for arbitrary \(\psi \in \Hfib(\X)\).
\end{proof}

\begin{corollary}
  \(\Hfib(\X)\) is compactly embedded in \(L^2(\X)\).
\end{corollary}
\begin{proof}
  This follows immediately from the fact that \(H^s(\X)\) is compactly embedded in \(L^2(\X)\) for any \(s > 0\).
\end{proof}

In particular, \(T\) is compact since it is continuous on \(L^2(\X)\). Hence, the following holds true:
\begin{corollary}\label{cor:pointwise-convergence}
  \(\norm{T - T_h} \to 0\) if and only if \(P_h \to 1\) pointwise on \(\range(T^*)\) and \(Q_h \to 1\) pointwise on \(\range(T)\).
\end{corollary}

Hence, under the assumptions of Theorem~\ref{thm:convergence-lu-flemming} and Corollary~\ref{cor:pointwise-convergence} and a parameter choice rule \(\alpha(\delta, h)\) satisfying \(\alpha(\delta, h) \geq \norm{T - T_h}^2\), it follows that
\begin{equation}
  \norm[\big]{\psi^\delta_{\alpha(\delta, h), h} - \psi^\dagger}_\fib \to 0 \quad\text{as \(\delta, h \to 0\)}.
\end{equation}

\section{Implementation}
\label{sec:implementation}

For the numerical implementation, we use the basis introduced in~\eqref{eq:basis-definition} above\footnote{More precisely, we use a slightly modified basis consisting of real-valued linear combinations of~\eqref{eq:basis-definition}.} and discretize by truncating to \({(\eta_{klm})}\) for \(\abs{k} \leq K_h\) and \(l \leq L_h\). The same basis is used in data space. Let \(P_h\) and \(Q_h\) denote the orthogonal projections onto \(X_h := \vspan\{\eta_{klm} \colon \abs{k} \leq K_h, l \leq L_h\}\) with respect to the norms \(\norm{\cdot}_\fib\) and \(\norm{\cdot}_{L^2}\), respectively. The pointwise convergence $\norm{P_h \psi - \psi}_\fib \to 0$ for all \(\psi \in \Hfib(\X)\) and $\norm{Q_h \psi - \psi}_\fib \to 0$ for all \(\psi \in L^2(\X)\) follows from Lemma~\ref{lemma:c-inf-dense-in-hfib} if \(K_h, L_h \to \infty\) as \(h \to 0\). We will need the Gramian matrix
\begin{equation}
  G_{klm, k'l'm'} = \inner{\eta_{klm}}{\eta_{k'l'm'}}_\fib
\end{equation}
for this basis. The computation of \(G\) requires evaluation of integrals of the form
\begin{equation}
  \int\limits_{S^2} u_i u_j Y_{lm}(u) \overline{Y_{l'm'}(u)} \,du.
\end{equation}
These can be calculated either by numerical quadrature, or using explicit --- albeit rather tedious --- formulas for Clebsch-Gordan coefficients.

We first numerically test the convergence of \(P_h\). Let \(\iota \colon \Hfib(\X) - L^2(\X)\) be the embedding. Then we want to investigate the norm convergence \(\norm{\iota \circ P_h \to \iota} \to 0\) as \(h \to 0\). The problem here is that for \(\psi \in \Hfib(\X)\), one needs to quantify \(\psi - P_h \psi\), i.e.\ its component outside of the discrete subspace. We approximate this by choosing a fine discretization \(X_0 := \vspan\{\eta_{klm} \colon \abs{k} \leq K_0, l \leq L_0\}\) for some sufficiently large \(K_0\), \(L_0\) and use this in place of the infinite-dimensional \(\Hfib(\X)\). Then \(P_h\) and the norm above can be evaluated numerically by taking advantage of the fact that the Gramian matrices of \(X_0\) and \(X_h\) decouple with respect to the spatial part \(\xi_k\) of the basis functions. In Figure~\ref{fig:convergence}, the results for various \(L_h\) are plotted against the spatial frequency cutoff \(K_h\) for norms with and without the additional Laplace-Beltrami operator, clearly showing convergence almost independently of the chosen \(L_h\). The figure also shows the results of the same numerical test when omitting the Laplace-Beltrami penalty. While these also decrease for large enough \(K_h\), the effect becomes weaker as the SH order \(L_h\) increases. This suggests that the observed decrease can be seen as an artifact of the Spherical Harmonics themselves introducing a smoothing effect (i.e.\ regularization by discretization), and hence convergence speed with respect to \(K_h\) deteriorates with increasing \(L_h\). Moreover, this might explain why good results can be observed even without additional explicit angular regularization.

\begin{figure}
  \centering
  \input{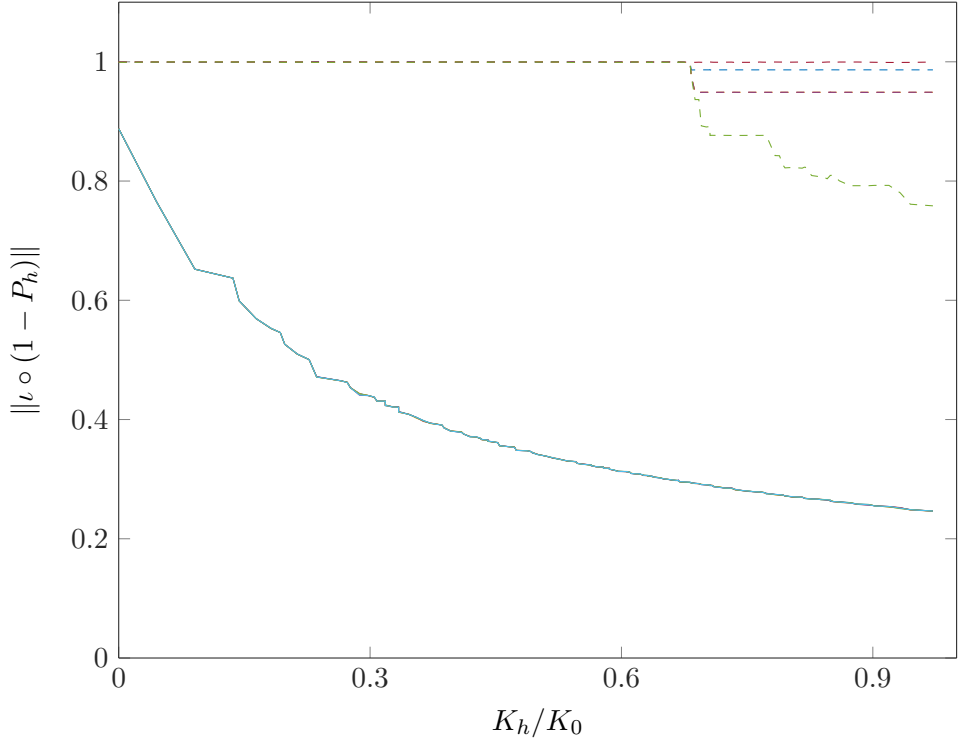}
  \caption{Convergence of the projection \(P_h\) viewed as a map \(\Hfib(\X) \to L^2(\X)\) (\(\iota\) is the embedding) plotted against the spatial cutoff frequency for even SH orders \(10 \leq L_h \leq 20\). \emph{Solid}: Including the angular regularization term in~\eqref{eq:def-fib-norm} (the curves are almost indistinguishable). \emph{Dashed}: Without the angular regularization term; SH order is increasing from bottom up.}\label{fig:convergence}
\end{figure}

For ODF reconstruction, the projected ODF and the forward operator are expanded in the discrete basis as
\begin{equation}
  P_h \psi = \sum_{klm} z_{klm} \eta_{klm}
\end{equation}
and
\begin{equation}
  T_h \eta_{klm} = Q_h T \eta_{klm} = \sum_{k'l'm'} {(B_h)}_{k'l'm', klm} \eta_{k'l'm'}.
\end{equation}
For the spherical convolution model, \(B_h\) is diagonal. Similarly, the discrete data is expanded as
\begin{equation}
  Q_h S^\delta = \sum_{klm} y^\delta_{klm} \eta_{klm}.
\end{equation}
Note that this can only be approximated in practice.

A method to implement the constrained spherical deconvolution problem~\eqref{eq:discrete-optimization-problem} was proposed in~\cite{tournier-et-al-2007}. In the following, we will describe this method and interpret it as a semi-smooth Newton method. This will prove local convergence of the method by general convergence results for semi-smooth Newton methods.

Instead of requiring \(P_h \psi \geq 0\), the constraint is checked only on a finite subset \({\{(x_i, u_j)\}}_{ij} \subset \X\),
\begin{equation}
  {(H z)}_{ij} := \sum_{klm} z_{klm} \eta_{klm}(x_i, u_j) \geq 0.
\end{equation}
The points \(x_i\) are usually chosen as the spatial grid on which the data \(S^\delta\) are given, while \(\{u_j\}\) is some point set on the sphere obtained e.g.\ by subdividing platonic solids. The constraint is only implemented approximately by an iteration of the form
\begin{equation}\label{eq:tournier-step}
  z_{k+1} = \argmin_z \left(\norm[\big]{B_h z - y^\delta}^2 + \alpha \norm[\big]{\sqrt{G} z}^2 + c \norm[\big]{\theta(- H z_k) \cdot H z}^2 \right),
\end{equation}
where \(\theta\) is the Heaviside step function and \(\cdot\) denotes element-wise multiplication. The constraint is approximated by a quadratic penalty on the set where the previous iterate violated it, i.e.\ \(H z_k < 0\). The parameter \(c > 0\) determines how strongly the constraint is enforced. This method is equivalent to applying the semi-smooth Newton method from~\cite{hintermueller-ito-kunisch} to the relaxed problem
\begin{equation}\label{eq:moreau-yosida}
  \argmin_{H z \geq w} \left(\norm[\big]{B_h z - y^\delta}^2 + \alpha \norm[\big]{\sqrt{G} z}^2 + c \norm[\big]{w}^2 \right).
\end{equation}
The equivalence can be seen by noting that the necessary and sufficient first order optimality conditions for~\eqref{eq:moreau-yosida} can be written as (see the reference cited above for more details)
\begin{equation}
  \begin{gathered}
    (B_h^* B_h + \alpha G) z - H^* \lambda = B_h^* y^\delta \\
    c w + \lambda = 0 \\
    \lambda = \max(0, \lambda - c (H z - w)),
  \end{gathered}
\end{equation}
which is the same as
\begin{equation}
  (B_h^* B_h + G) z + c H^* \min(0, H z) = B_h^* y^\delta.
\end{equation}
Finally,~\eqref{eq:tournier-step} is precisely a semi-smooth Newton step for this equation. The minimization problem~\eqref{eq:moreau-yosida} can be interpreted as Moreau-Yosida regularization of the constraint \(H z \geq 0\). Convergence of both the semi-smooth Newton method for \(k \to \infty\) and the Moreau-Yosida regularization for \(c \to \infty\) have been investigated in~\cite{hintermueller-ito-kunisch,ito-kunisch}. For the numerical experiments, we will use this method since it is efficient and easy to implement. It allows for comparison of the regularized method with the original CSD method, which is included as the case \(\alpha = 0\).

\section{Numerical experiments}
\label{sec:numerics}

Performance of the method was tested on a physical phantom and an \emph{in vivo} measurement. Both data sets were obtained by other groups.

Construction of the phantom and data acquisition is described in more detail in~\cite{fillard-et-al}. The data was acquired on a \(64 \times 64 \times 3\)-grid with \(\SI{3}{mm}\) isotropic voxel spacing, using 64 gradient directions and \(b = \SI{1500}{s/mm^2}\). The structure of the phantom is shown in figure~\ref{fig:fc}, together with three regions of interest --- two crossings and a curved structure. Reconstructions were performed on the whole data set, without employing a ``white'' matter mask. The SH coefficients of the convolution kernel up to order 8 were estimated from the data using the \emph{MRtrix} software package.\footnote{Available at \url{http://www.brain.org.au/software/}. Developed at Brain Research Institute, Melbourne, Australia. Further information can also be found in in~\cite{tournier-et-al-2012}.} The ODF was reconstructed in Spherical Harmonics up to order 12. The parameter \(c\) was simply chosen as \(c = 1\). Figure~\ref{fig:fc-compare} shows reconstruction results for these regions using four regularization strategies: unregularized constrained spherical deconvolution (CSD), deconvolution with fiber continuity penalty (CSD-FC), with additional angular Laplace-Beltrami penalty (CSD-FC+LB), and CSD with isotropic spatial regularization (CSD-iso). For the crossing regions, the spatial penalties clearly improve spatial coherence and resolution of the crossing compared to the unregularized reconstruction. Compared to the isotropic penalty, the FC penalty significantly reduces artifacts perpendicular to the fibers. For the curved structure, the FC penalty causes some visible artifacts tangential to the structure. These are reduced somewhat by the additional angular penalty, essentially by blurring the spurious crossings. This of course also limits the achievable angular resolution. It also produces some artifacts perpendicular to the fibers since the embedding into \(H^{1/2}(\X)\) implies some degree of isotropic smoothness. Therefore, the corresponding regularization parameter should not be chosen too large.

\begin{figure}
  \centering
  \includegraphics[width=.4\textwidth]{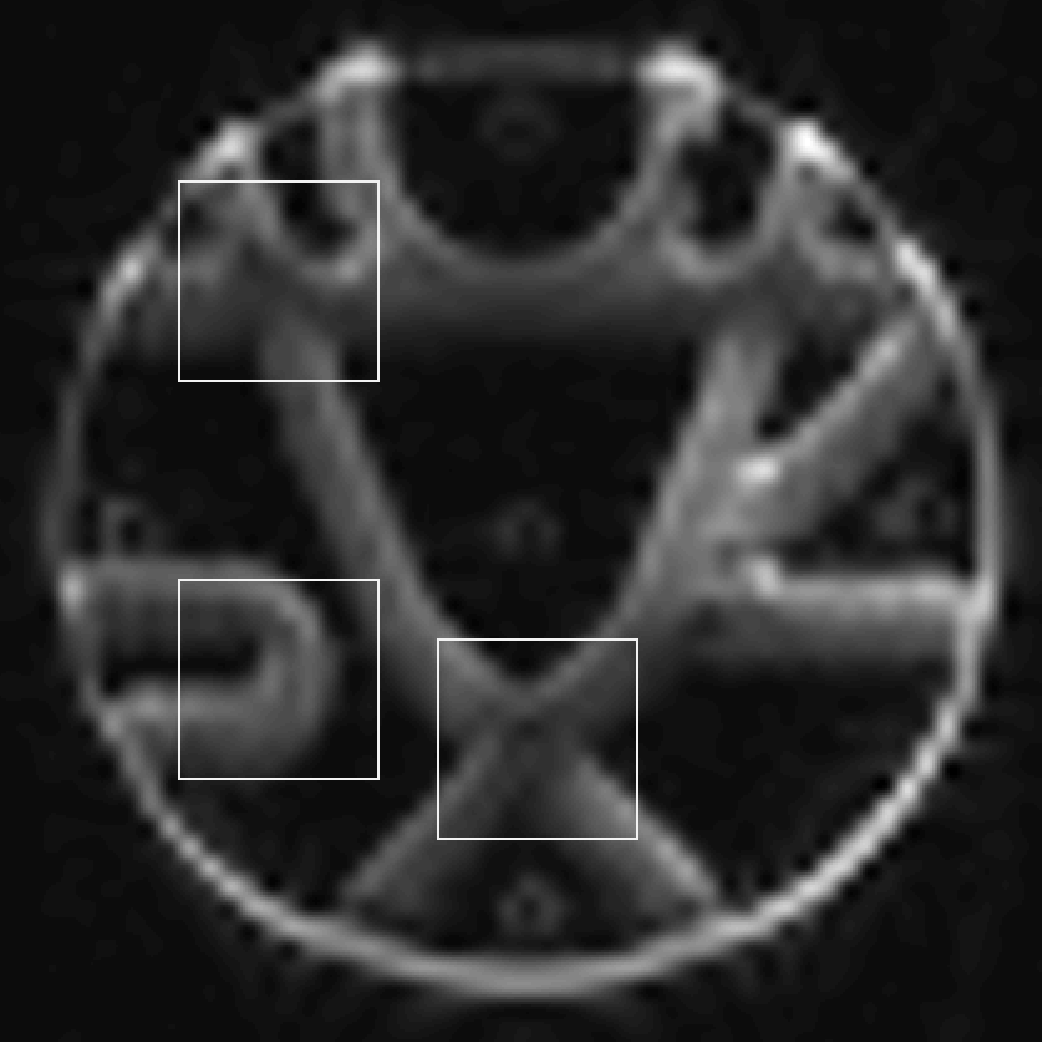}
  \caption{Structure of the physical phantom (\(b = 0\) image). Reconstructions for the highlighted regions are depicted in figure~\ref{fig:fc-compare}.}\label{fig:fc}
\end{figure}

\begin{figure}
  \centering
  \begin{tikzpicture}[node distance=0.5em,inner sep=0pt]
    \node (l2 cross1)
    {\includegraphics[width=.3\textwidth]{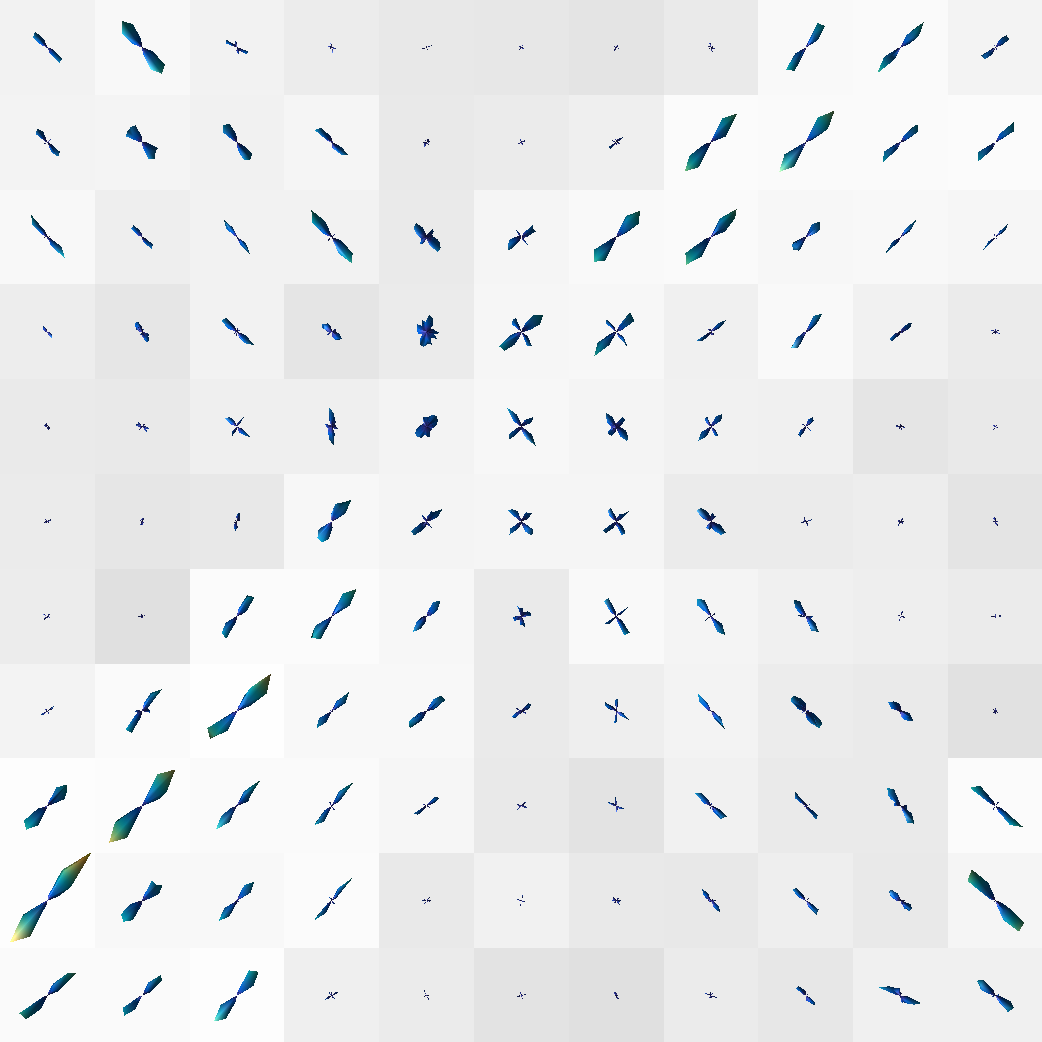}};
    \node (l2 curve) [right=of l2 cross1.east,anchor=west]
    {\includegraphics[width=.3\textwidth]{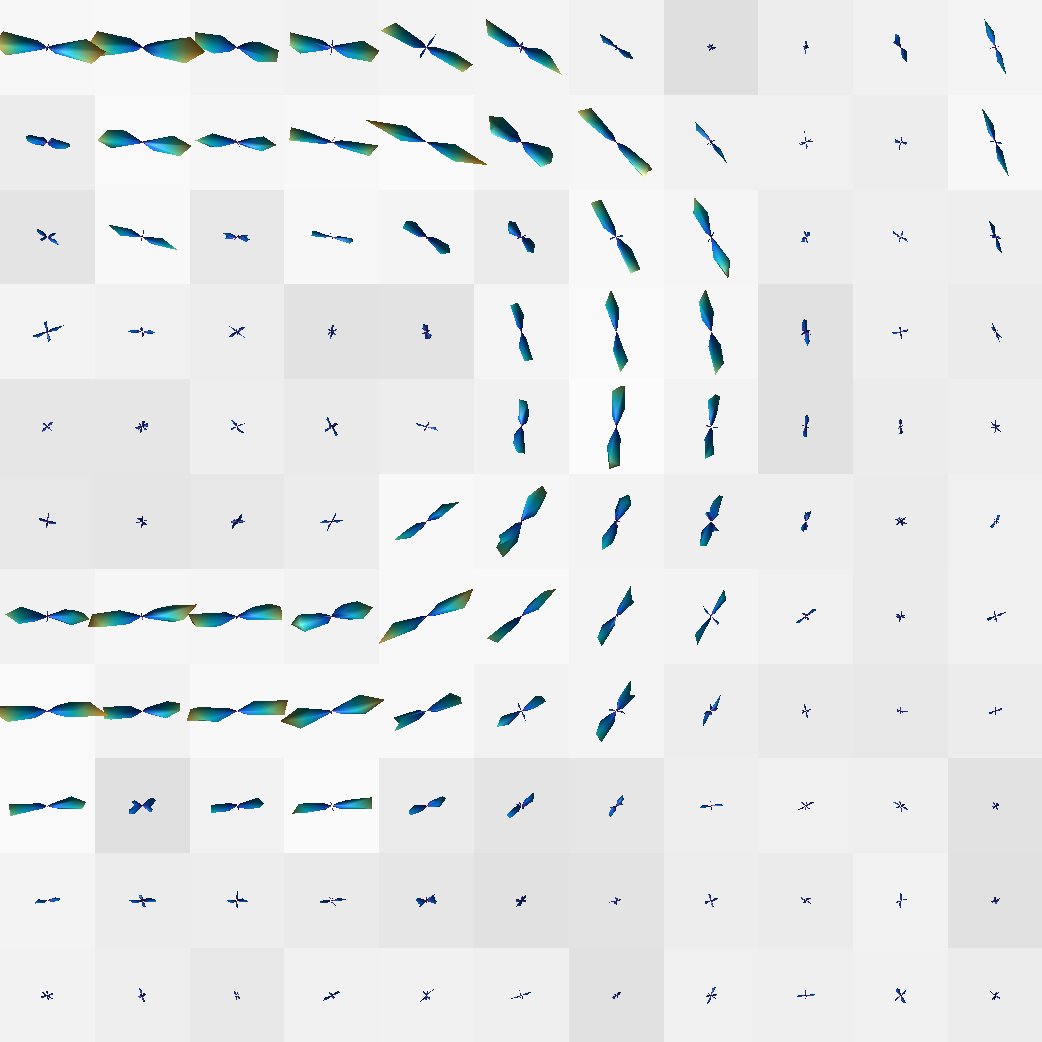}};
    \node (l2 cross2) [right=of l2 curve.east,anchor=west]
    {\includegraphics[width=.3\textwidth]{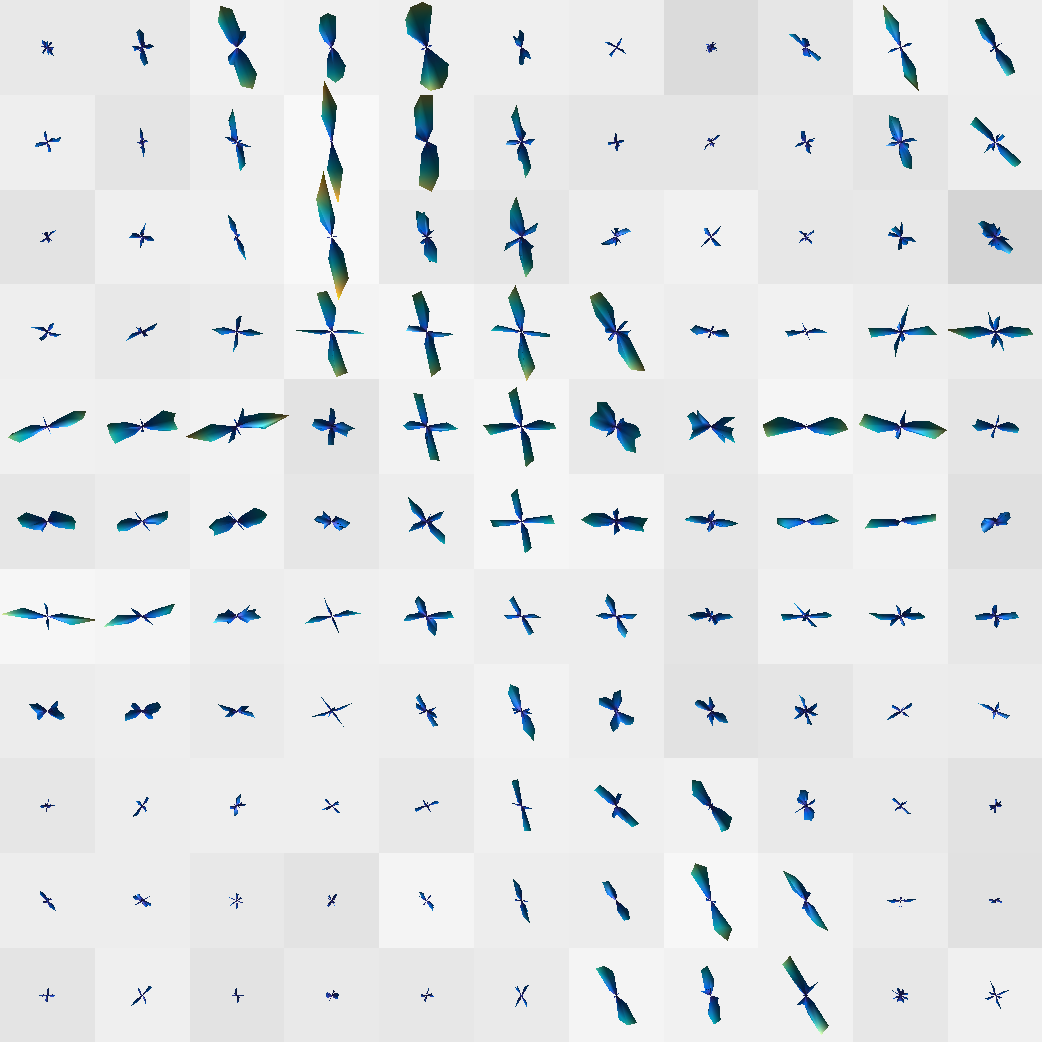}};
    \node [rotate=90,left=of l2 cross1.west,anchor=south] {CSD};

    \node (l2+spatial cross1) [below=of l2 cross1.south,anchor=north]
    {\includegraphics[width=.3\textwidth]{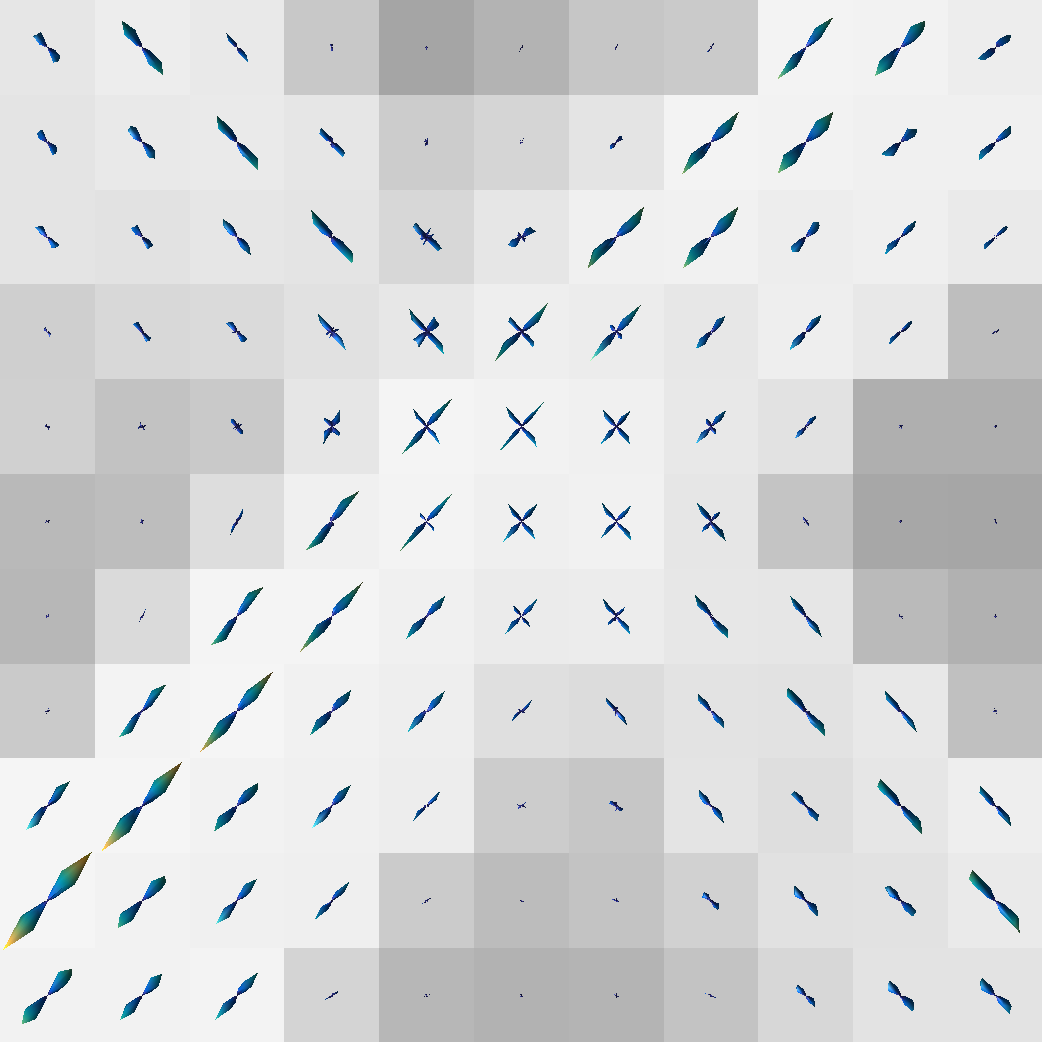}};
    \node (l2+spatial curve) [right=of l2+spatial cross1.east,anchor=west]
    {\includegraphics[width=.3\textwidth]{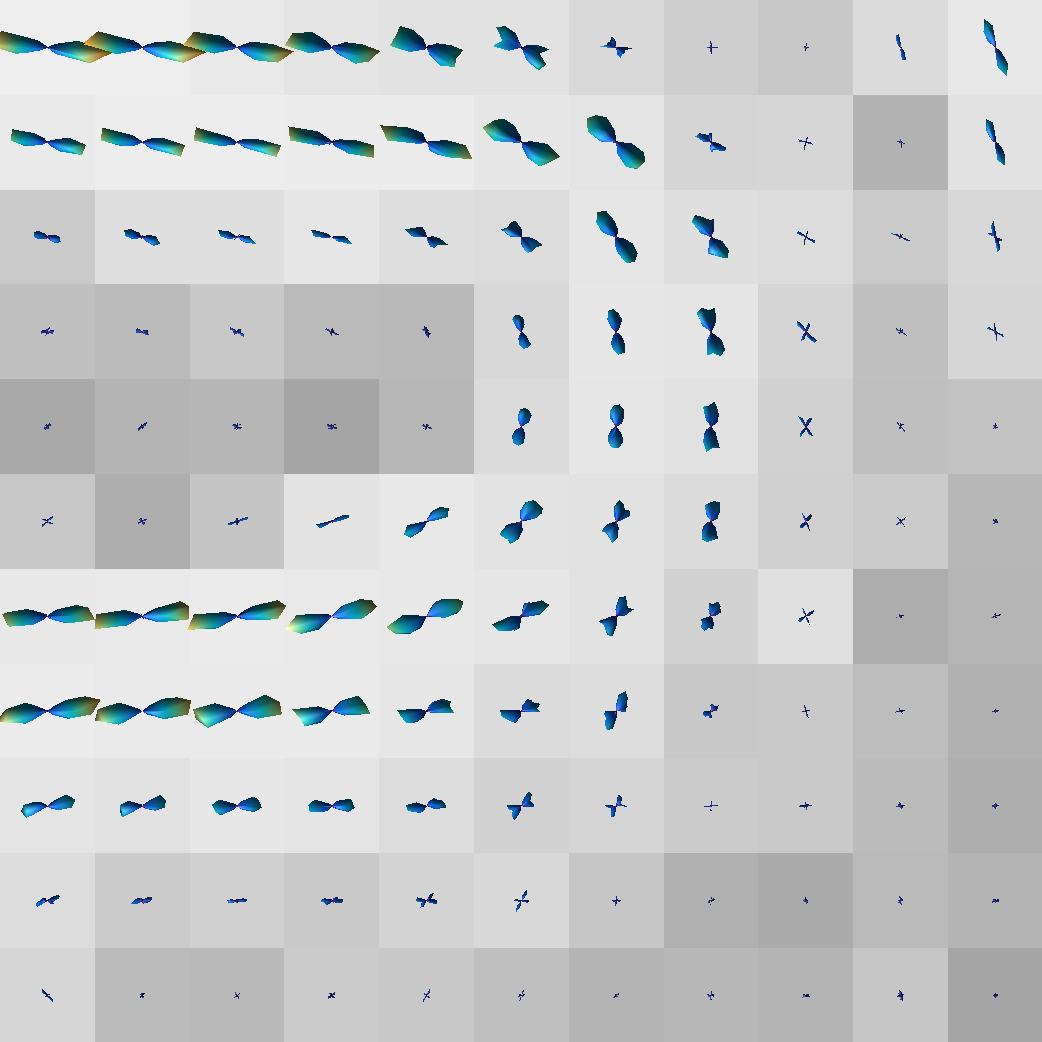}};
    \node (l2+spatial cross2) [right=of l2+spatial curve.east,anchor=west]
    {\includegraphics[width=.3\textwidth]{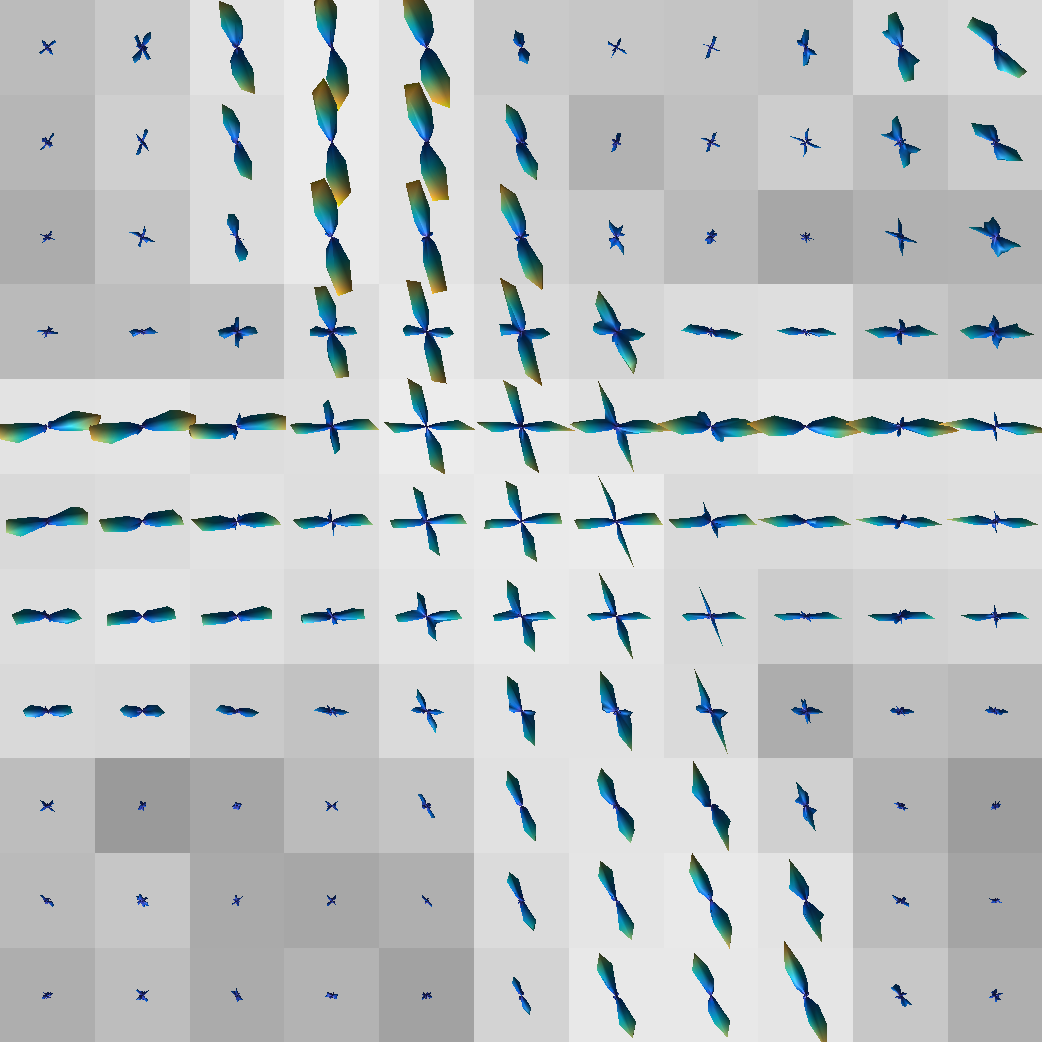}};
    \node [rotate=90,left=of l2+spatial cross1.west,anchor=south] {CSD-FC};

    \node (angular+spatial cross1) [below=of l2+spatial cross1.south,anchor=north]
    {\includegraphics[width=.3\textwidth]{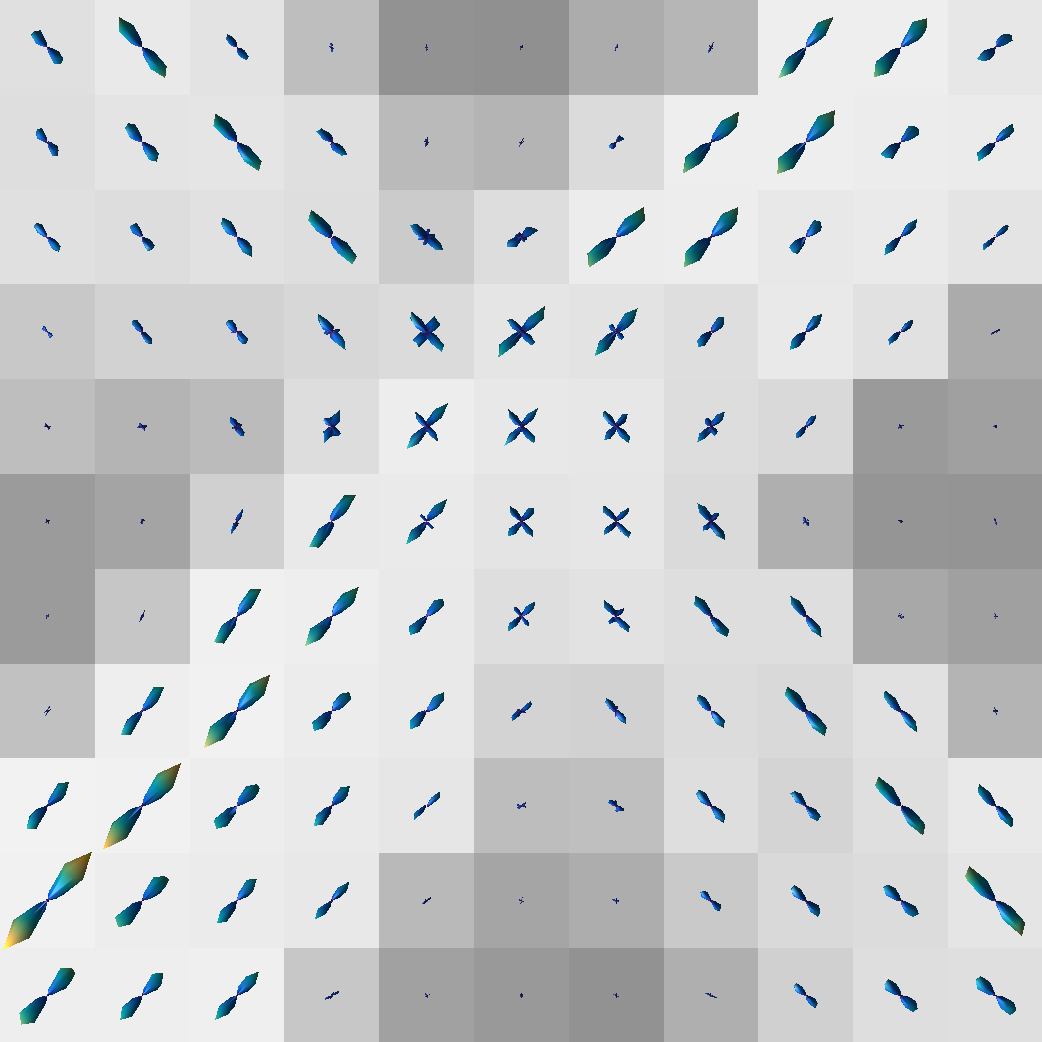}};
    \node (angular+spatial curve) [right=of angular+spatial cross1.east,anchor=west]
    {\includegraphics[width=.3\textwidth]{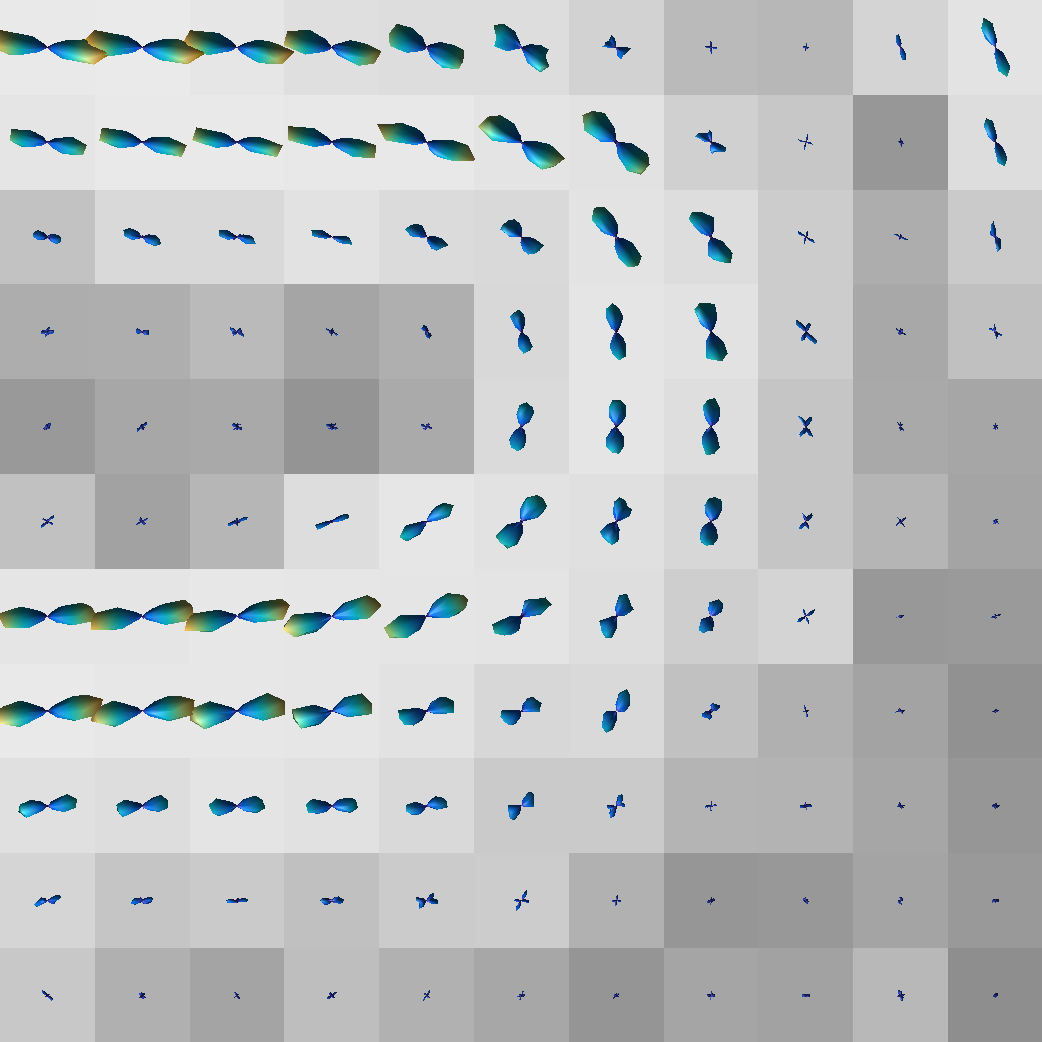}};
    \node (angular+spatial cross2) [right=of angular+spatial curve.east,anchor=west]
    {\includegraphics[width=.3\textwidth]{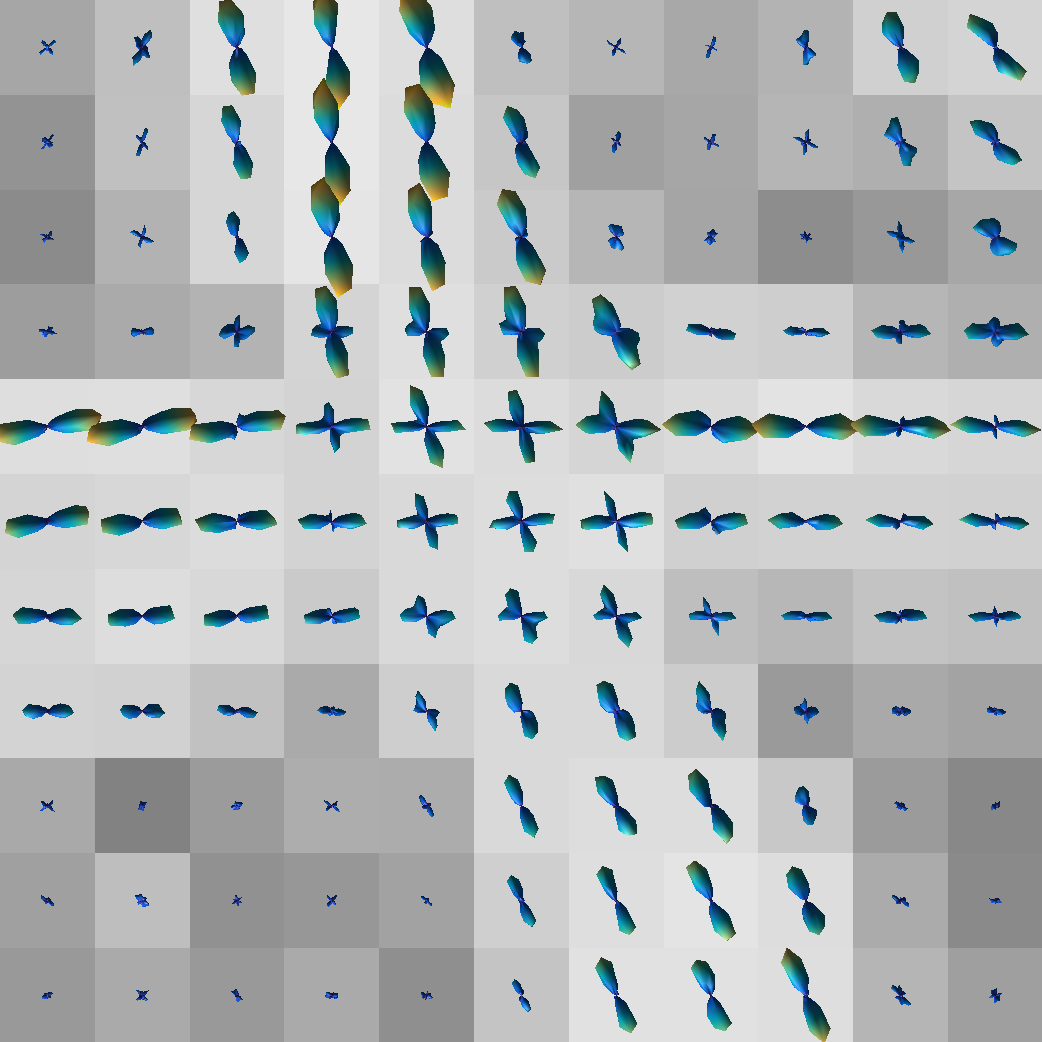}};
    \node [rotate=90,left=of angular+spatial cross1.west,anchor=south] {CSD-FC+LB};

    \node (iso cross1) [below=of angular+spatial cross1.south,anchor=north]
    {\includegraphics[width=.3\textwidth]{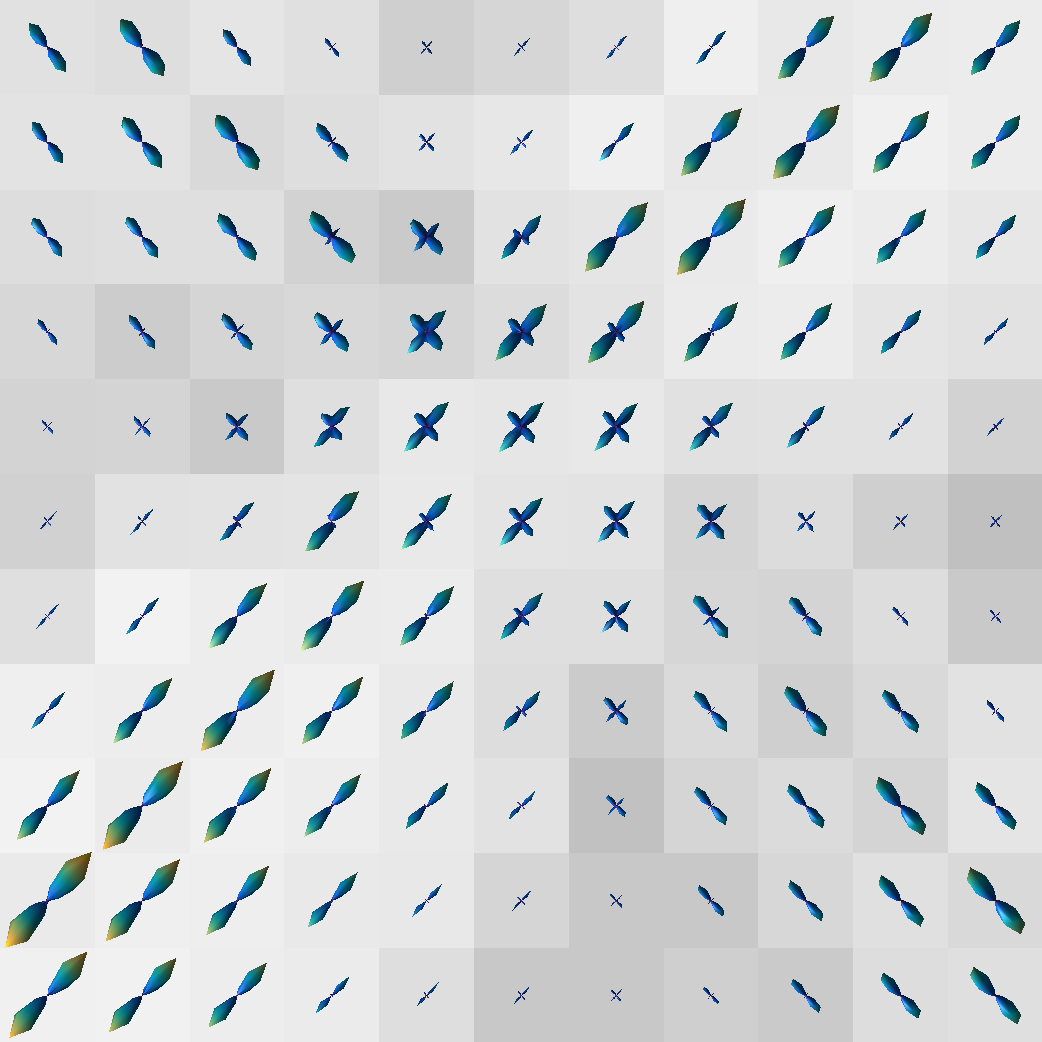}};
    \node (iso curve) [right=of iso cross1.east,anchor=west]
    {\includegraphics[width=.3\textwidth]{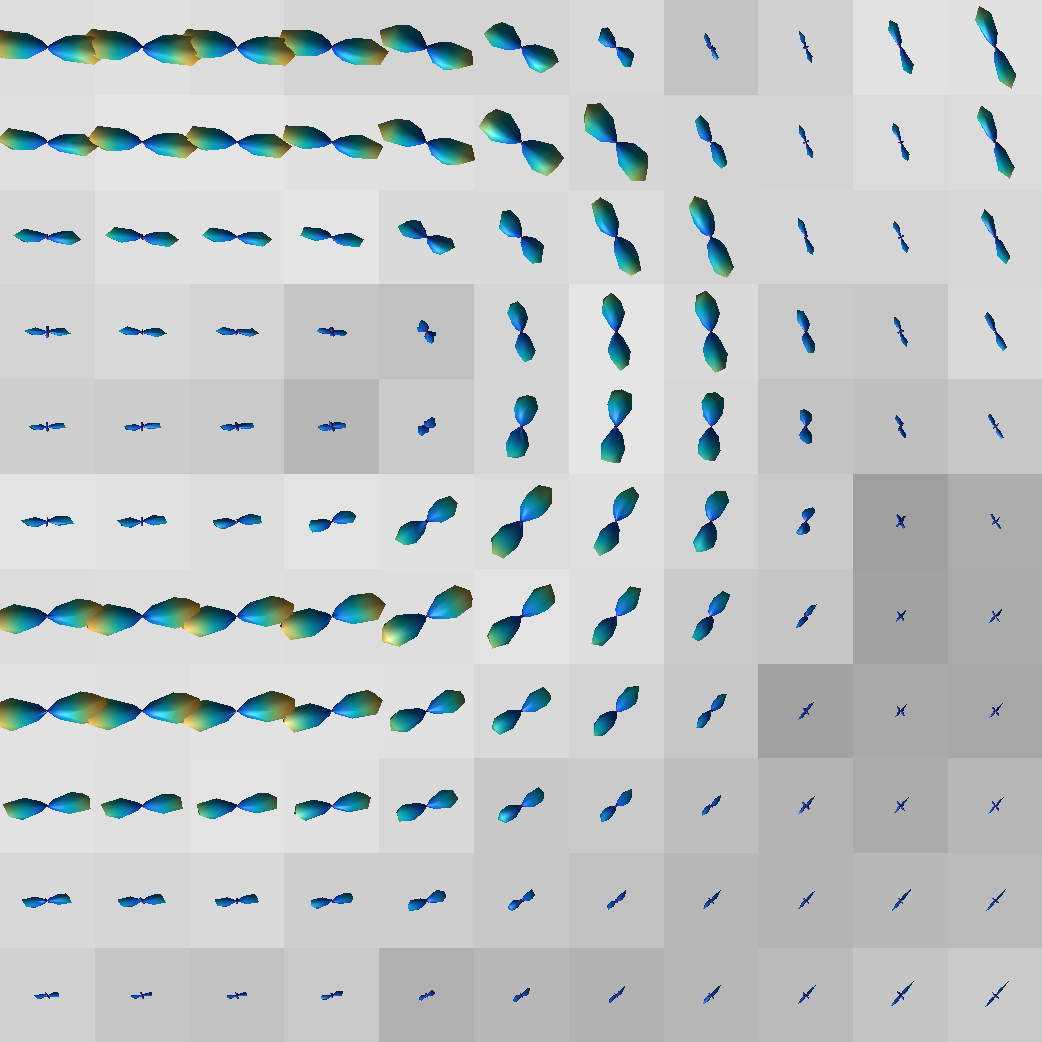}};
    \node (iso cross2) [right=of iso curve.east,anchor=west]
    {\includegraphics[width=.3\textwidth]{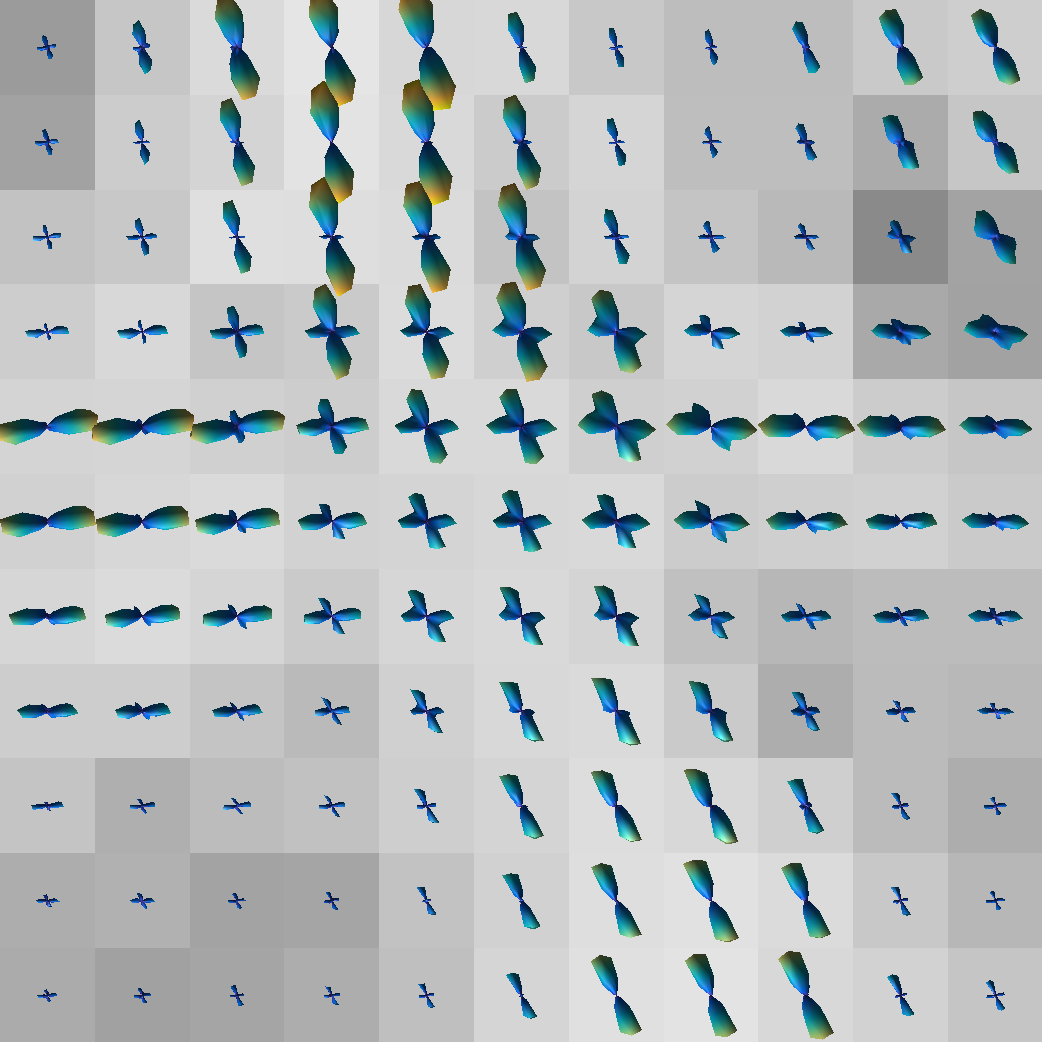}};
    \node [rotate=90,left=of iso cross1.west,anchor=south] {CSD-iso};
  \end{tikzpicture}
  \caption{Performance of different regularization strategies on the highlighted regions in figure~\ref{fig:fc}. From top to bottom: (1) Constrained Spherical Deconvolution, (2) CSD with FC penalty (3), CSD with FC and Laplace-Beltrami penalty, (4) CSD with isotropic spatial penalty. The background image shows the (generalized) FA.}\label{fig:fc-compare}
\end{figure}
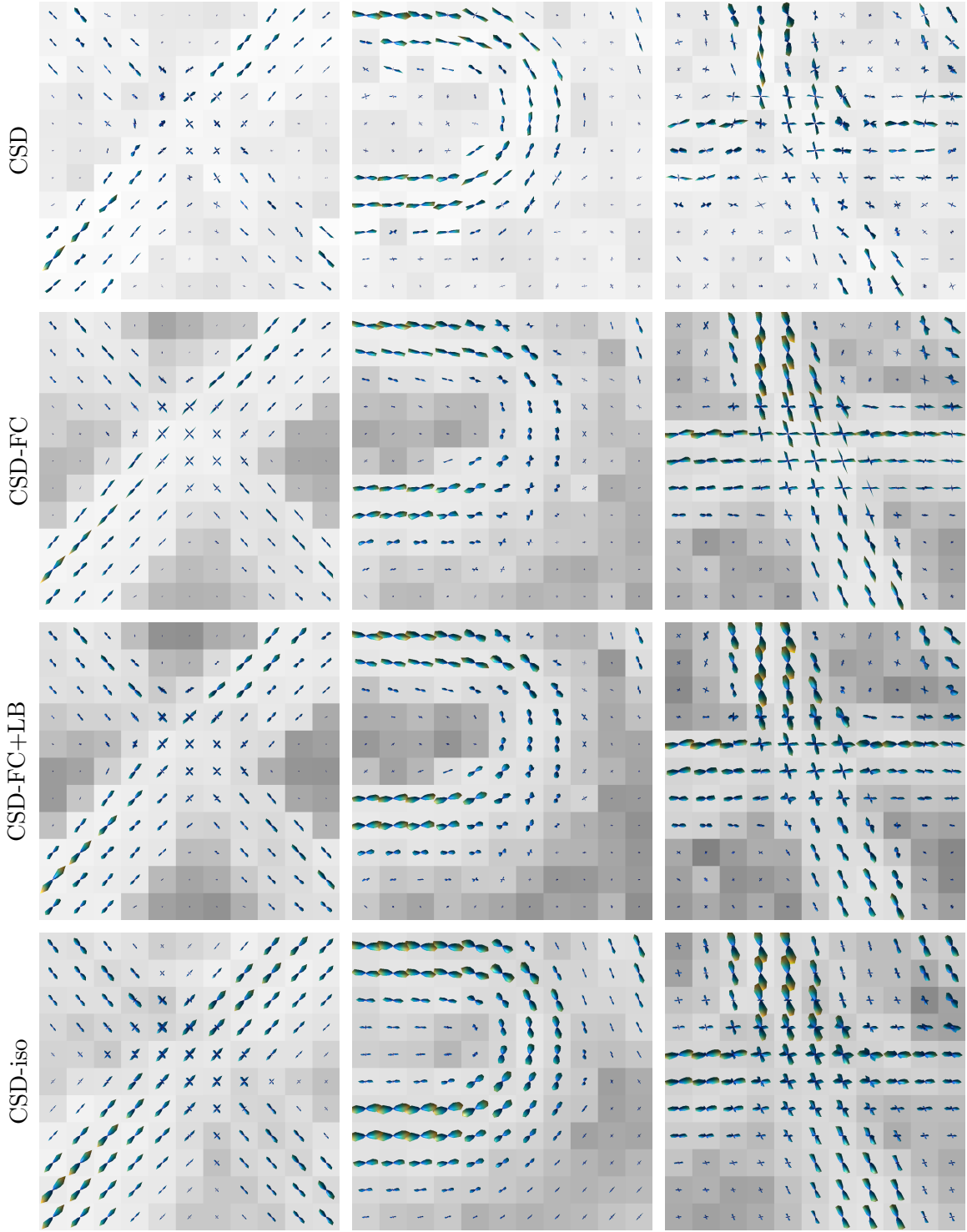

The \emph{in vivo} data set was taken from the Human Connectome Project (HCP) database.\footnote{See \url{https://ida.loni.usc.edu} and the \emph{Acknowledgments} section.} The data set consists of each 90 diffusion weighted images for \(b\)-values \(1000\), \(2000\) and \(\SI{3000}{s/mm^2}\). Of these, we only used the \(b = \SI{2000}{s/mm^2}\) points. As the data set is intended primarily for brain research, not for evaluating reconstruction and tracking methods, its SNR is rather high, so that even unregularized reconstructions show a good spatial coherence. Still, it is interesting to test the performance the CSD-FC method on this data set, in particular to see how well the geometric assumption of locally straight fibers is fulfilled in a realistic situation. Therefore, the spatial regularization parameter was deliberately chosen rather large to highlight some of its strengths and shortcomings.

Reconstructions were performed on a \(30 \times 30 \times 35\) voxel subset of the volume around the area shown in figure~\ref{fig:hcp}. We parametrized the convolution kernel as an exponential function, \(k(t) \simeq \exp(-\alpha t^2)\), with parameters obtained from a previous DTI reconstruction by averaging the sorted eigenvalues of all tensors with fractional anisotropy (FA) exceeding a certain threshold and taking the mean value of the smaller two eigenvalues in order to obtain an axially symmetric kernel. Instead of an unregularized reconstruction for comparison, we used a small \(L^2\) penalty. Visual inspection showed that this did not have significant impact on the resolved structures. However, it increased FA contrast by reducing noise in empty areas and, more importantly, lead to ODFs that are slightly broader and can therefore be depicted in printed plots more clearly.

\begin{figure}
  \centering
  \includegraphics[width=.35\textwidth]{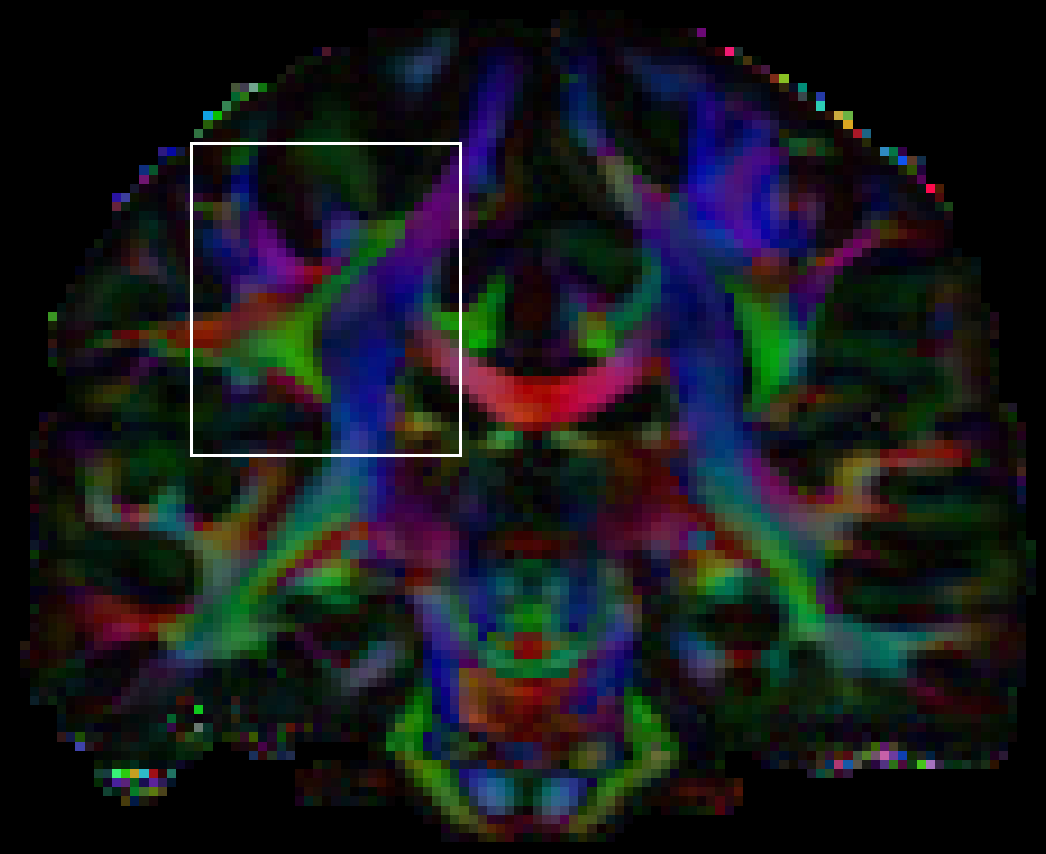}
  \caption{Subregion of the \emph{in vivo} on which reconstructions were performed (the image shows FA color coded by main diffusion direction). The depicted slice is the one in figure~\ref{fig:hcp08}, while figure~\ref{fig:hcp20} is a bit closer to the anterior.}\label{fig:hcp}
\end{figure}

Reconstruction results for two slices are shown in figures~\ref{fig:hcp08} and~\ref{fig:hcp20}. For brevity's sake, we only compare the \(L^2\)-penalized reconstruction to the CSD-FC+LB method.

\begin{figure}
  \centering
  \begin{tikzpicture}[node distance=0.5em,inner sep=0pt]
    \node (l2)
    {\includegraphics[height=17.5em]{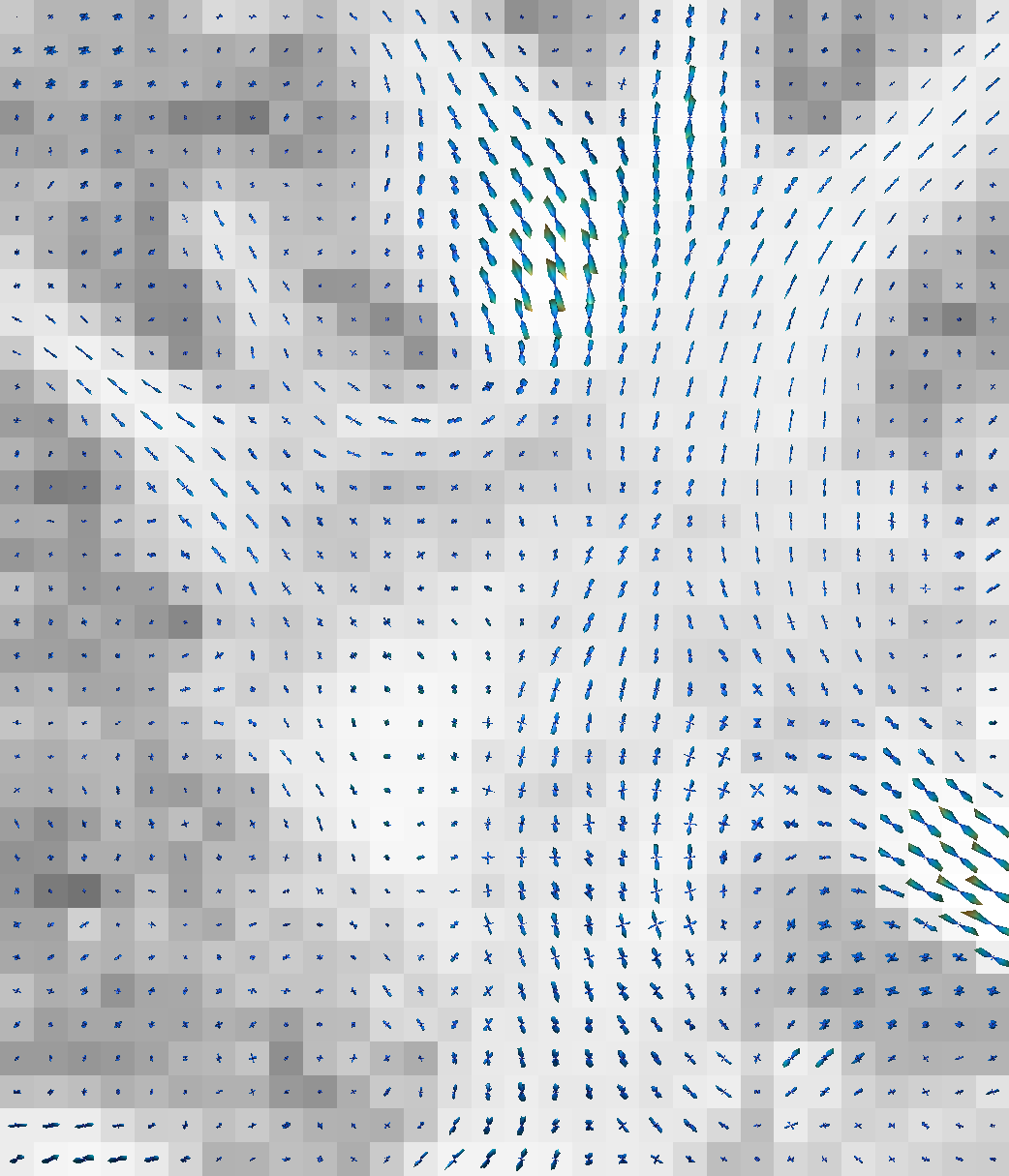}};

    \node [right=of l2.north east,anchor=north west]
    {\includegraphics[height=8.5em]{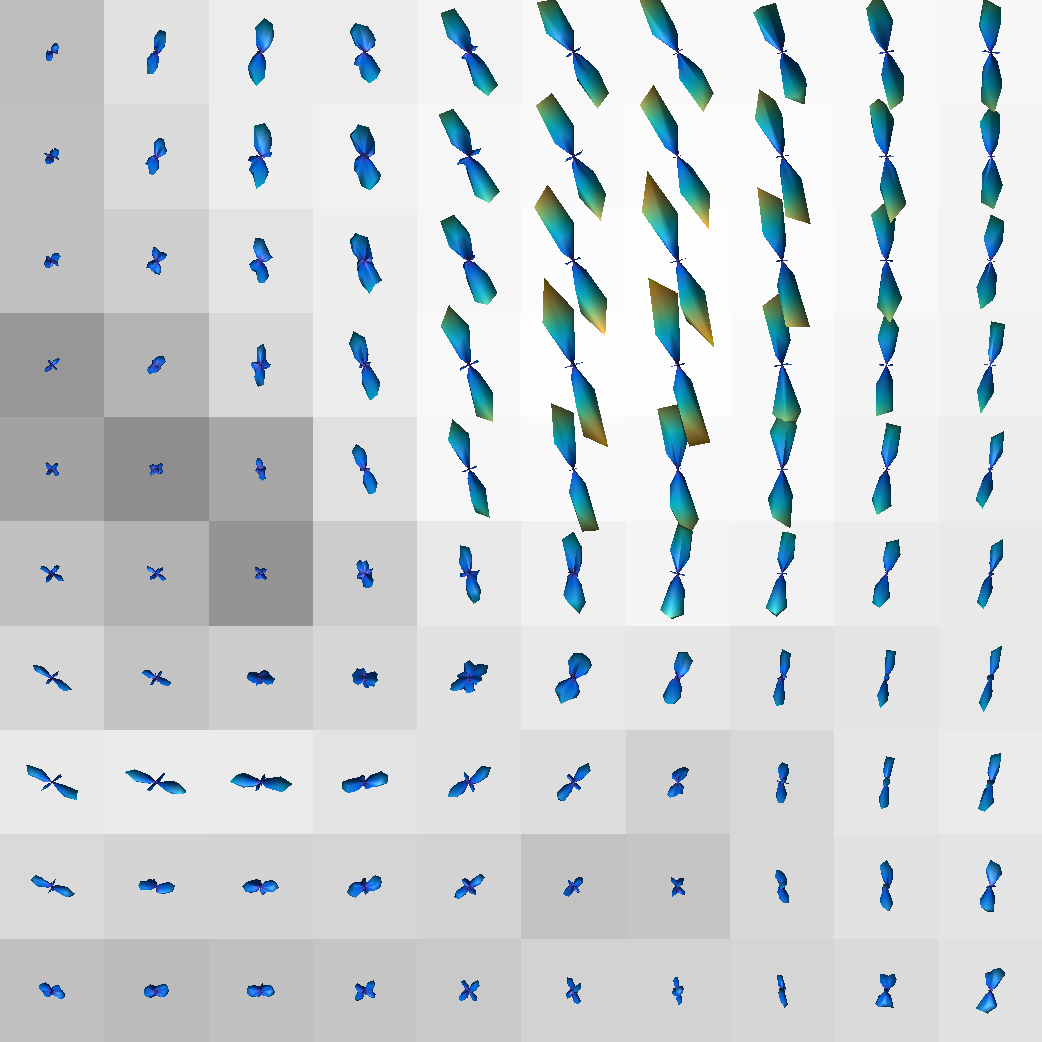}};
    \draw [thick] ([xshift=5em,yshift=10em] l2.south west) rectangle +(5em,5em);

    \node [right=of l2.south east,anchor=south west]
    {\includegraphics[height=8.5em]{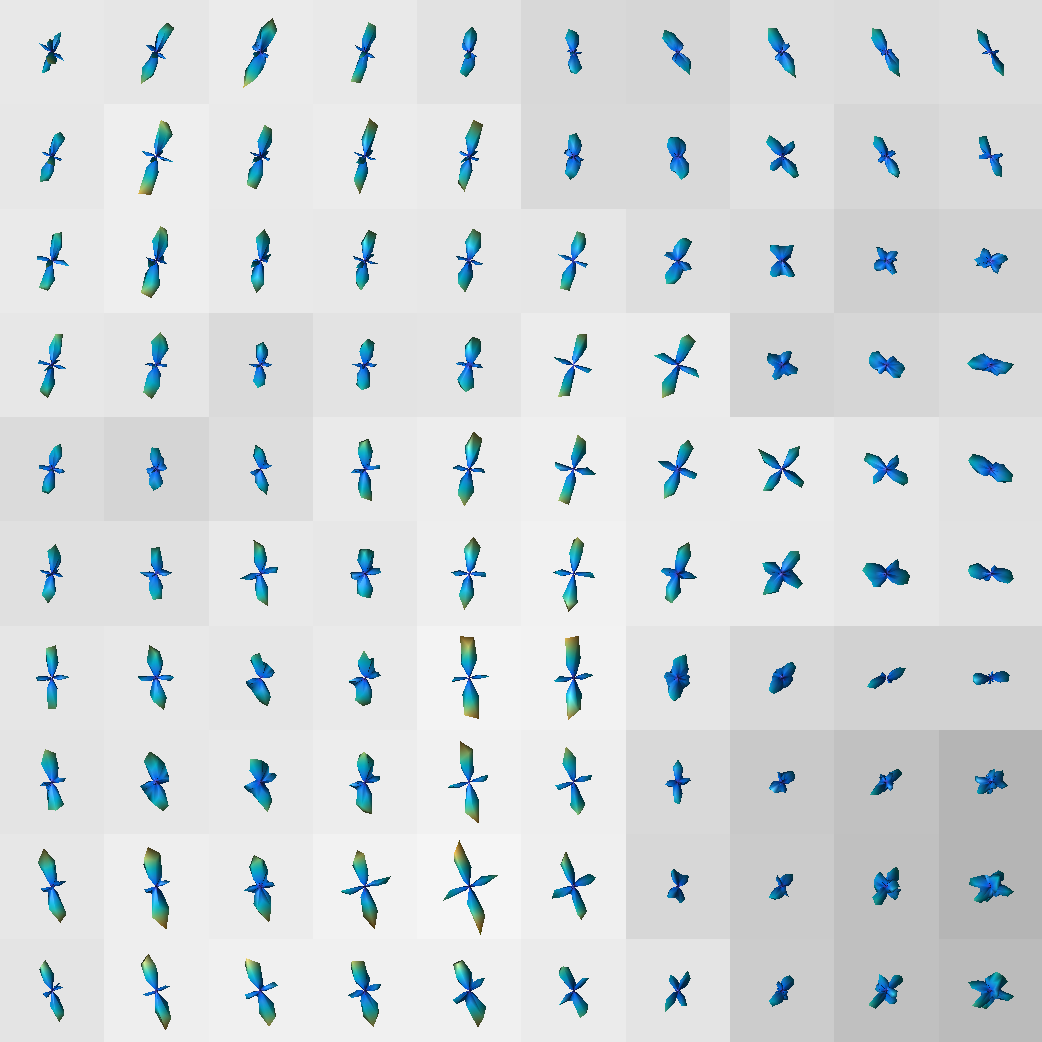}};
    \draw [thick] ([xshift=7.5em,yshift=3em] l2.south west) rectangle +(5em,5em);

    \node [rotate=90,left=of l2.west,anchor=south] {CSD-\(L^2\)};

    \node [below=of l2.south,anchor=north] (spatial)
    {\includegraphics[height=17.5em]{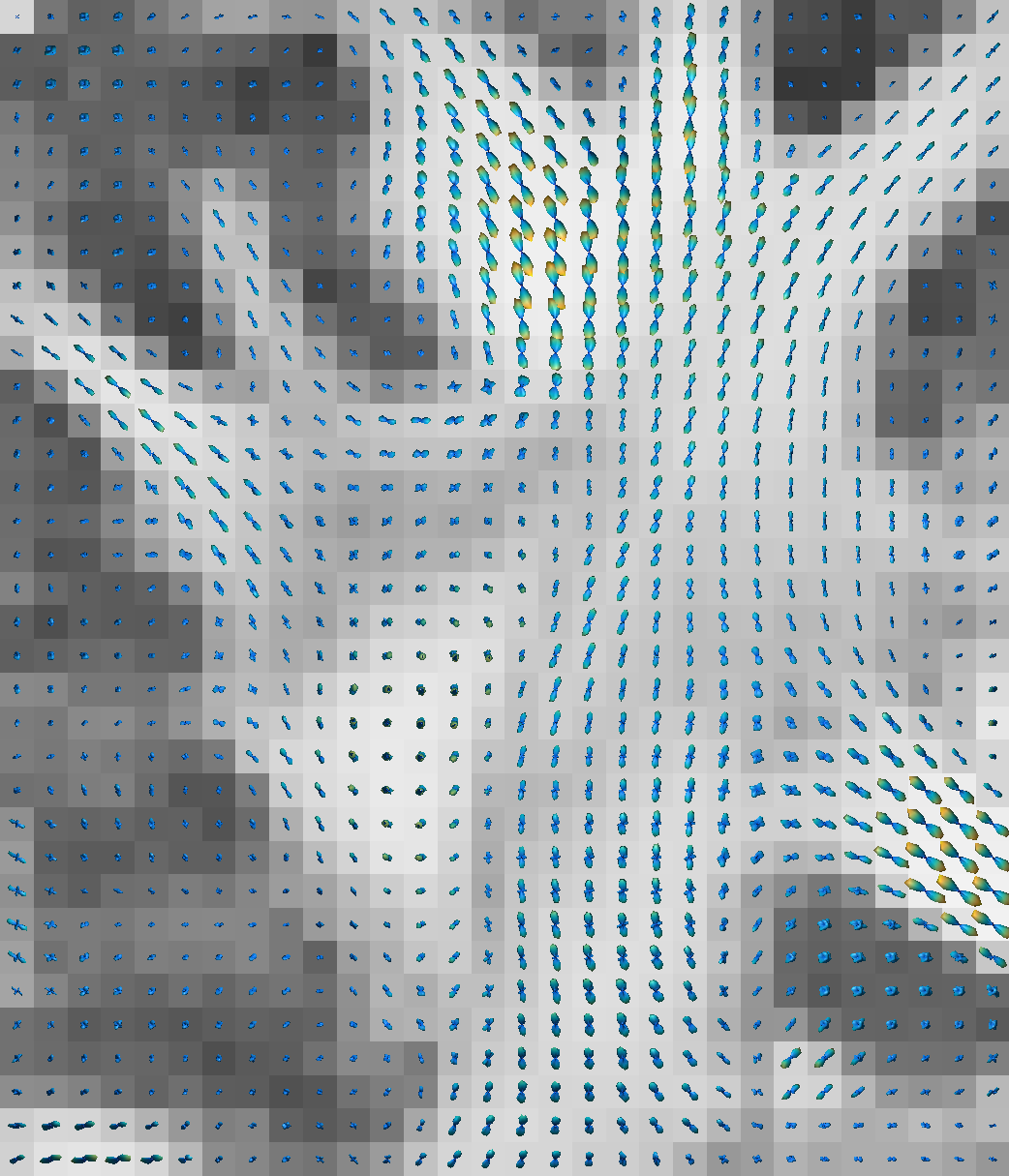}};

    \node [right=of spatial.north east,anchor=north west]
    {\includegraphics[height=8.5em]{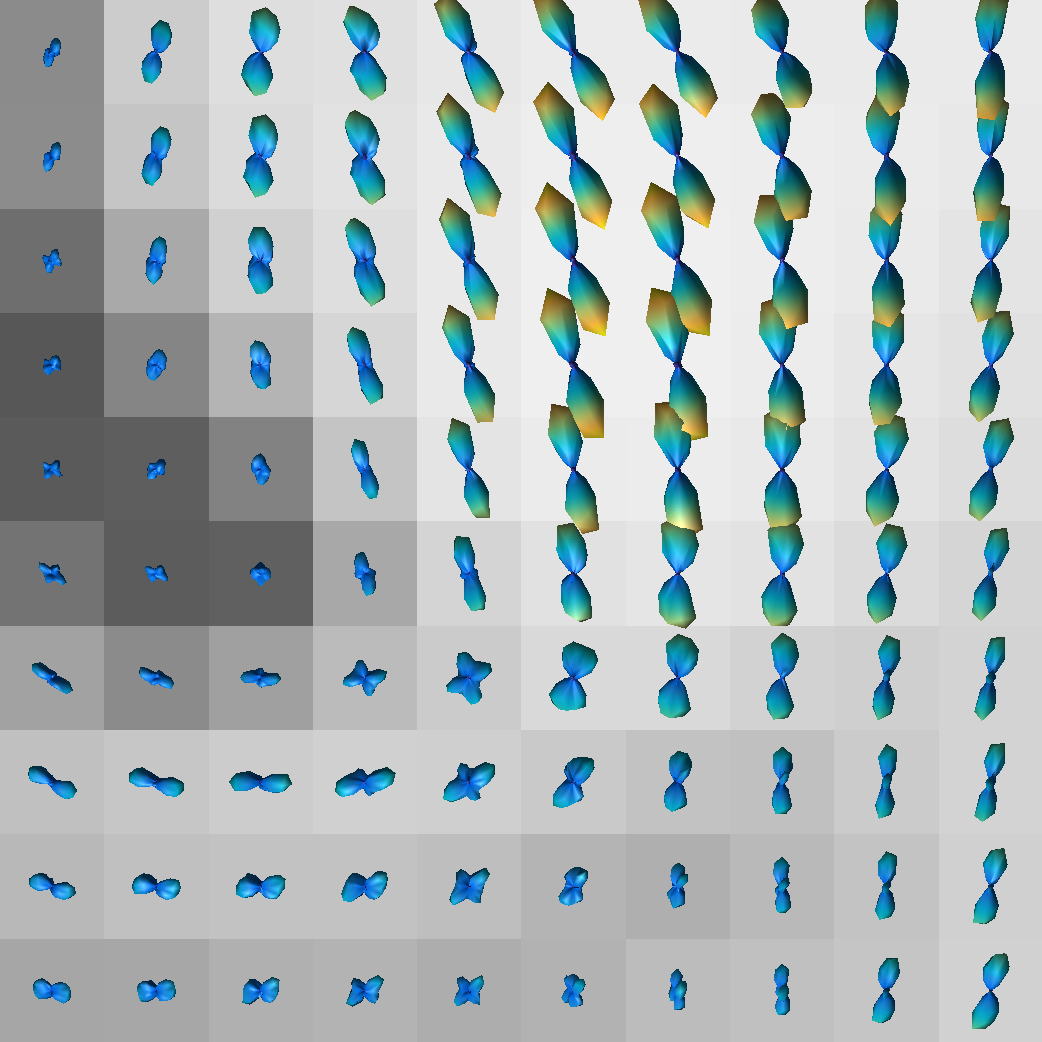}};
    \draw [thick] ([xshift=5em,yshift=10em] spatial.south west) rectangle +(5em,5em);

    \node [right=of spatial.south east,anchor=south west]
    {\includegraphics[height=8.5em]{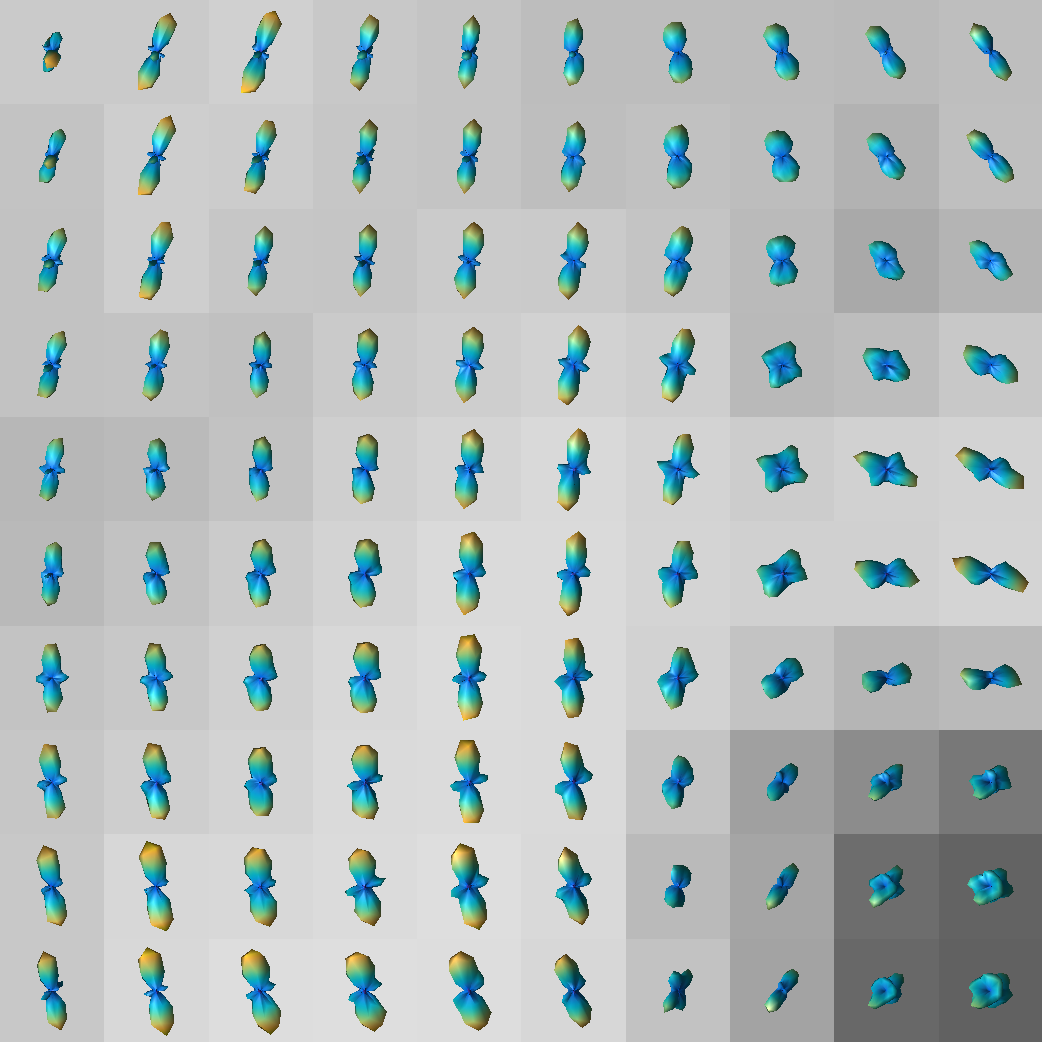}};
    \draw [thick] ([xshift=7.5em,yshift=3em] spatial.south west) rectangle +(5em,5em);

    \node [rotate=90,left=of spatial.west,anchor=south] {CSD-FC+LB};
  \end{tikzpicture}
  \caption{Comparison of CSD with \(L^2\) penalty (top) to CSD with FC and Laplace-Beltrami penalty (bottom) for \emph{in vivo} data shown in Figure~\ref{fig:hcp}.}\label{fig:hcp08}
\end{figure}
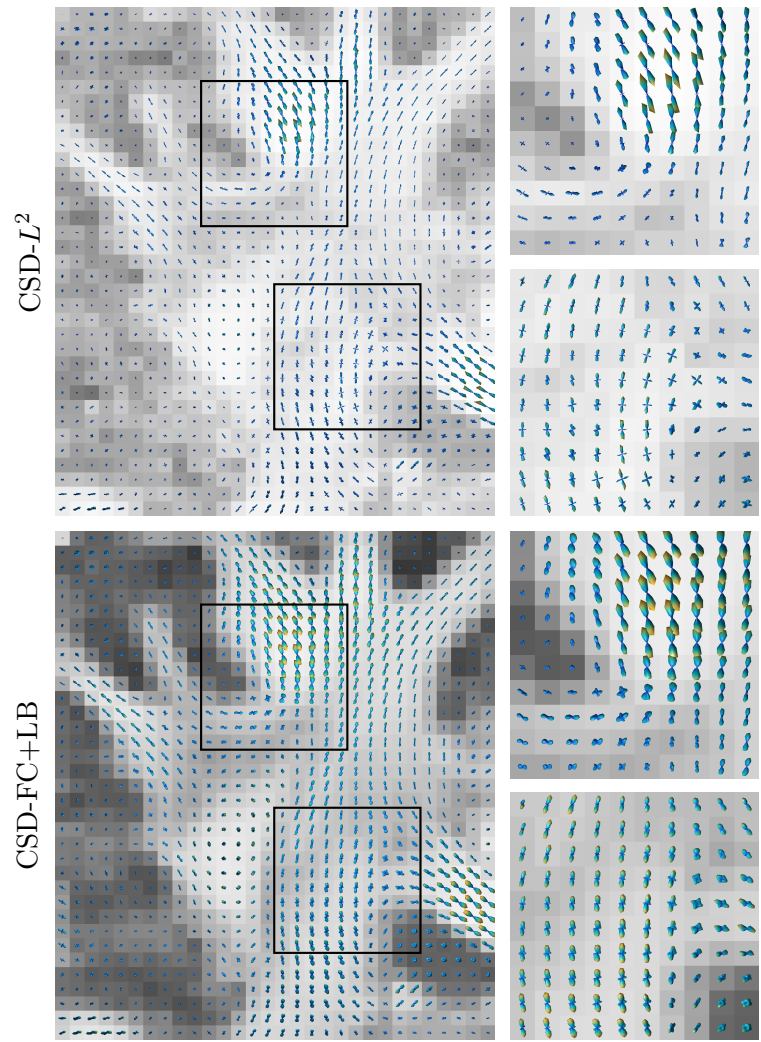

\begin{figure}
  \centering
  \begin{tikzpicture}[node distance=0.5em,inner sep=0pt]
    \node (l2)
    {\includegraphics[height=17.5em]{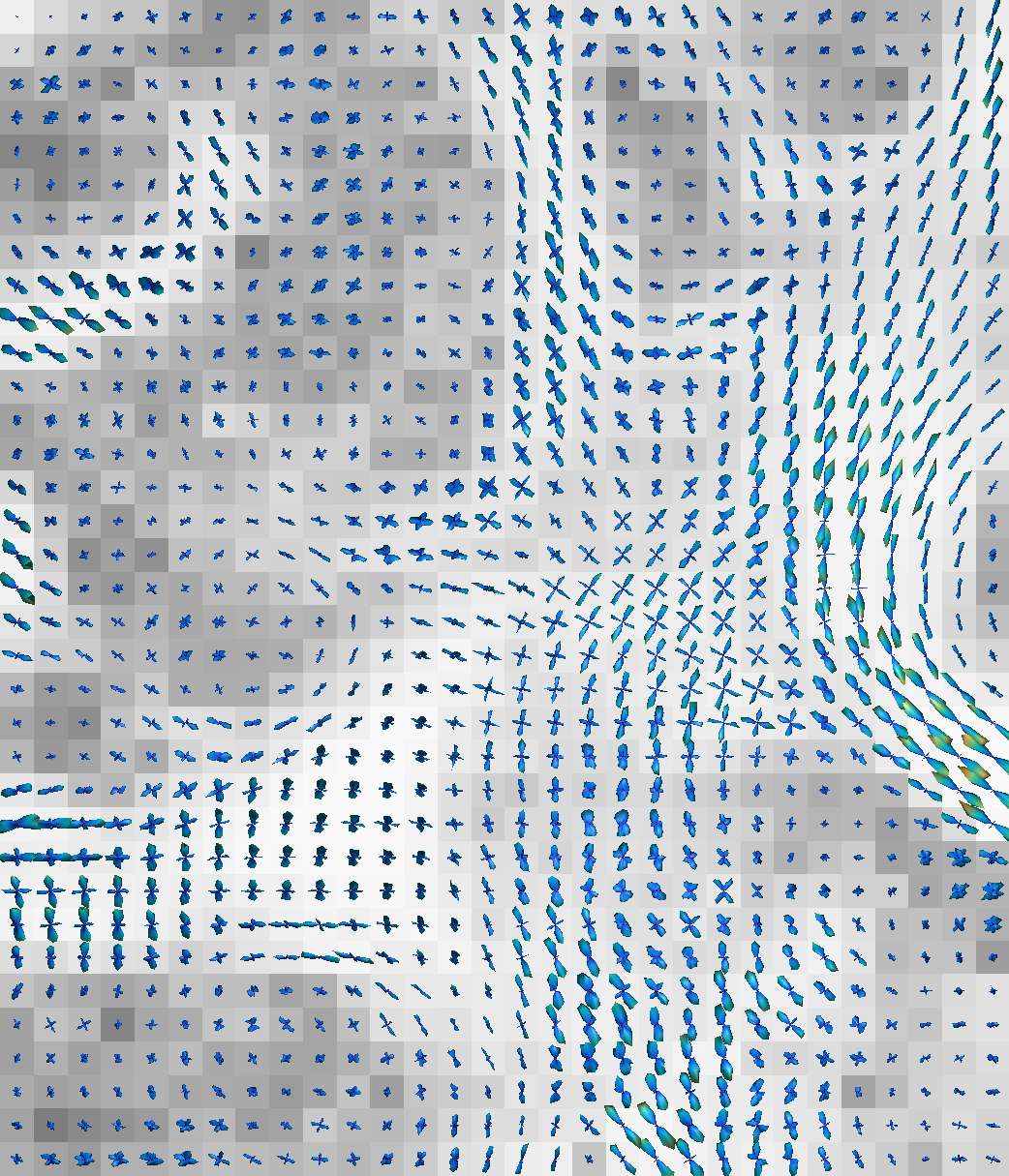}};

    \node [right=of l2.north east,anchor=north west]
    {\includegraphics[height=8.5em]{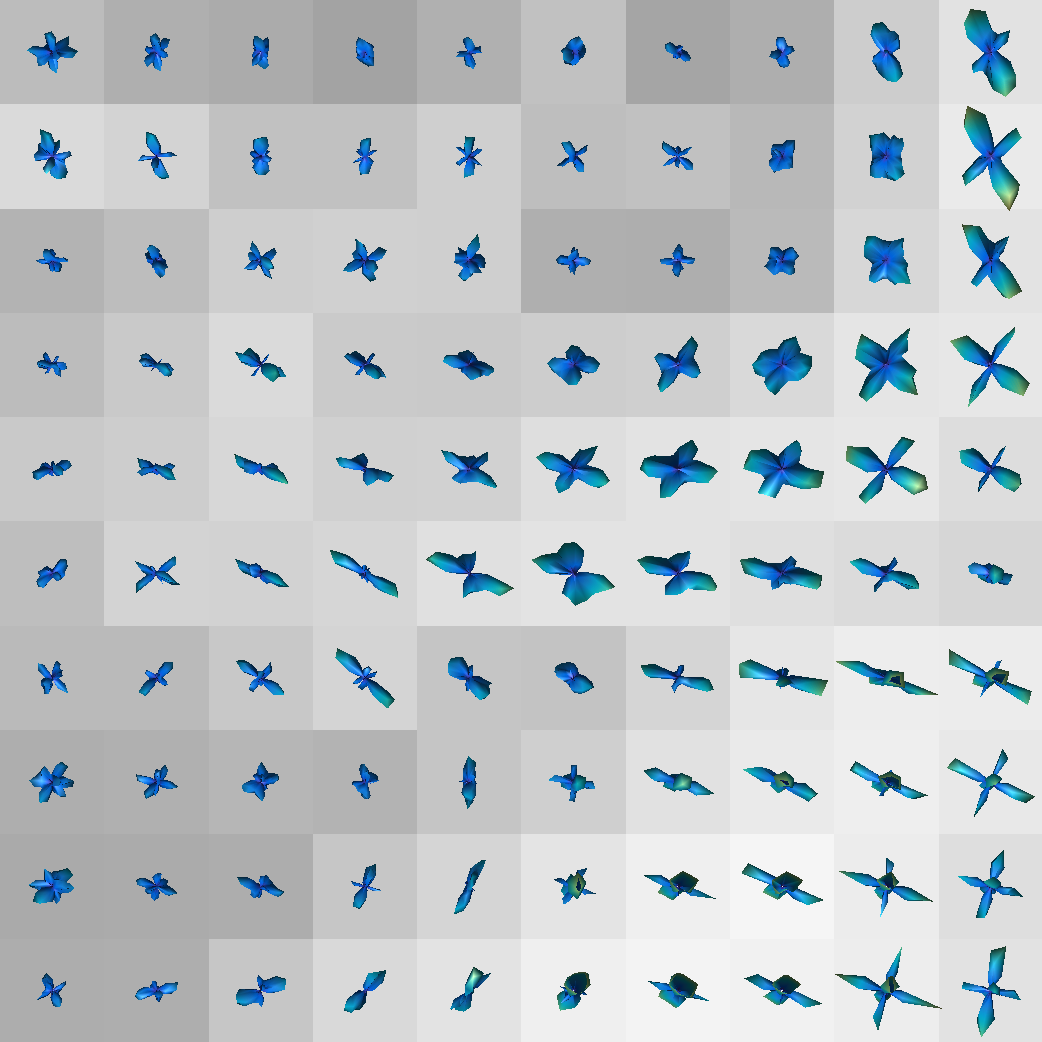}};
    \draw [thick] ([xshift=3em,yshift=7em] l2.south west) rectangle +(5em,5em);

    \node [right=of l2.south east,anchor=south west]
    {\includegraphics[height=8.5em]{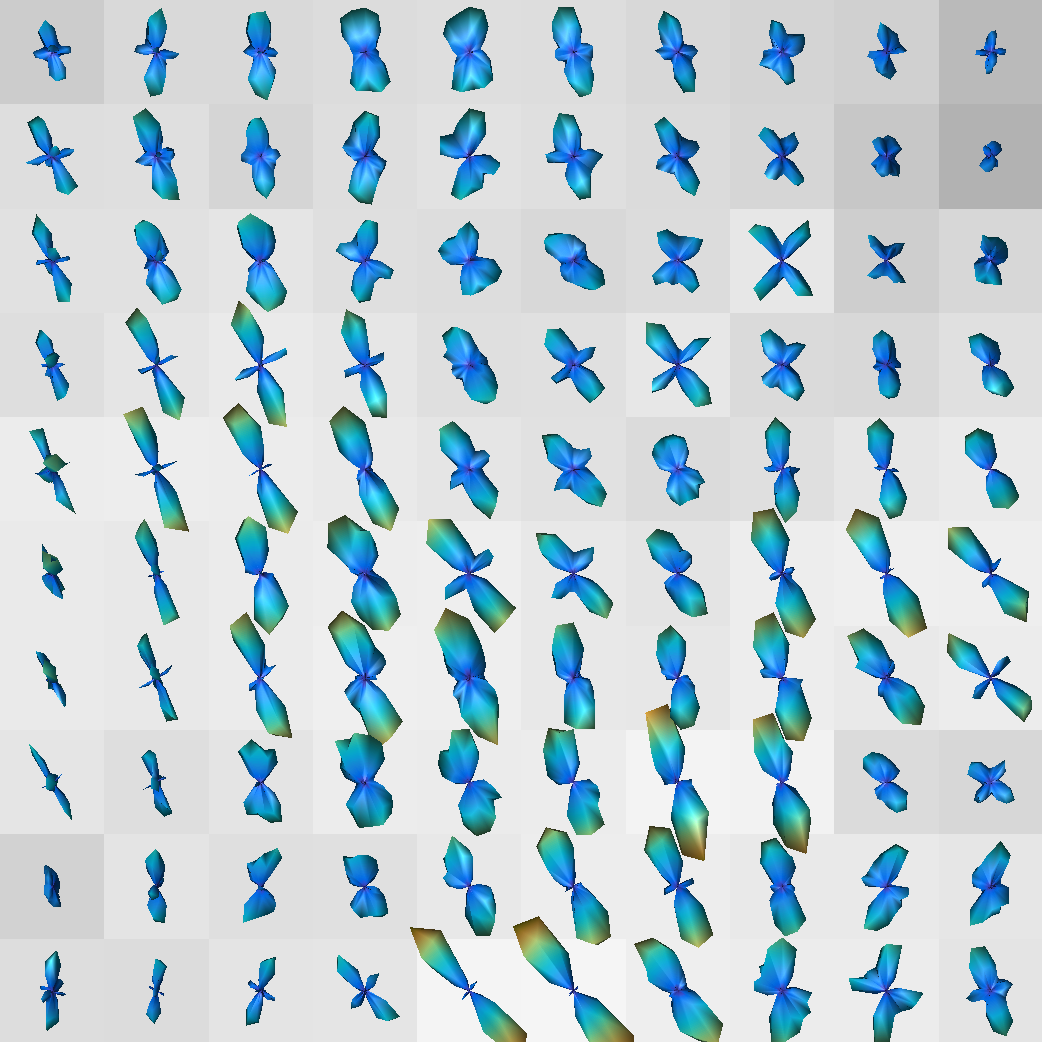}};
    \draw [thick] ([xshift=7em,yshift=0.5em] l2.south west) rectangle +(5em,5em);

    \node [rotate=90,left=of l2.west,anchor=south] {CSD-\(L^2\)};

    \node [below=of l2.south,anchor=north] (spatial)
    {\includegraphics[height=17.5em]{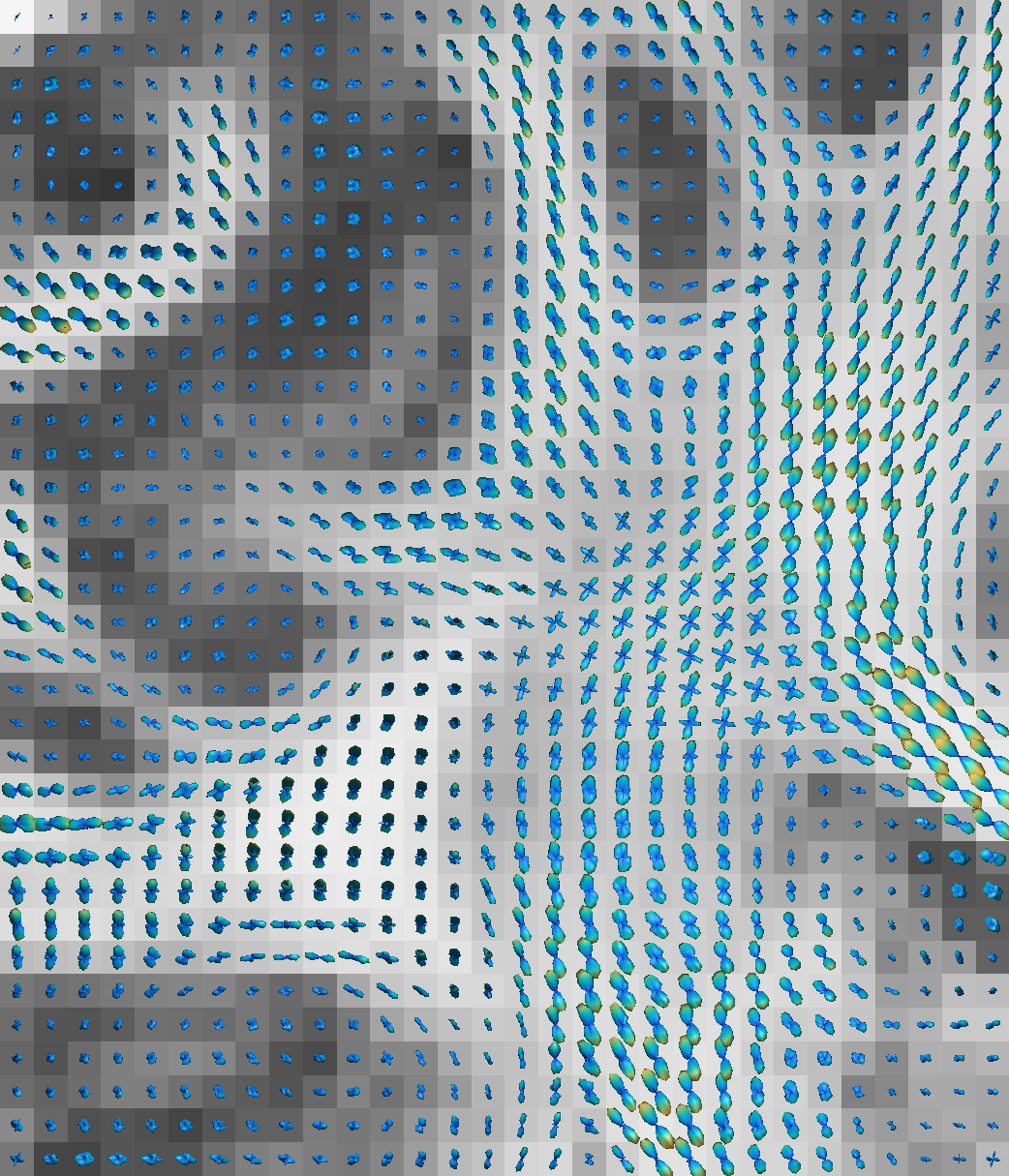}};

    \node [right=of spatial.north east,anchor=north west]
    {\includegraphics[height=8.5em]{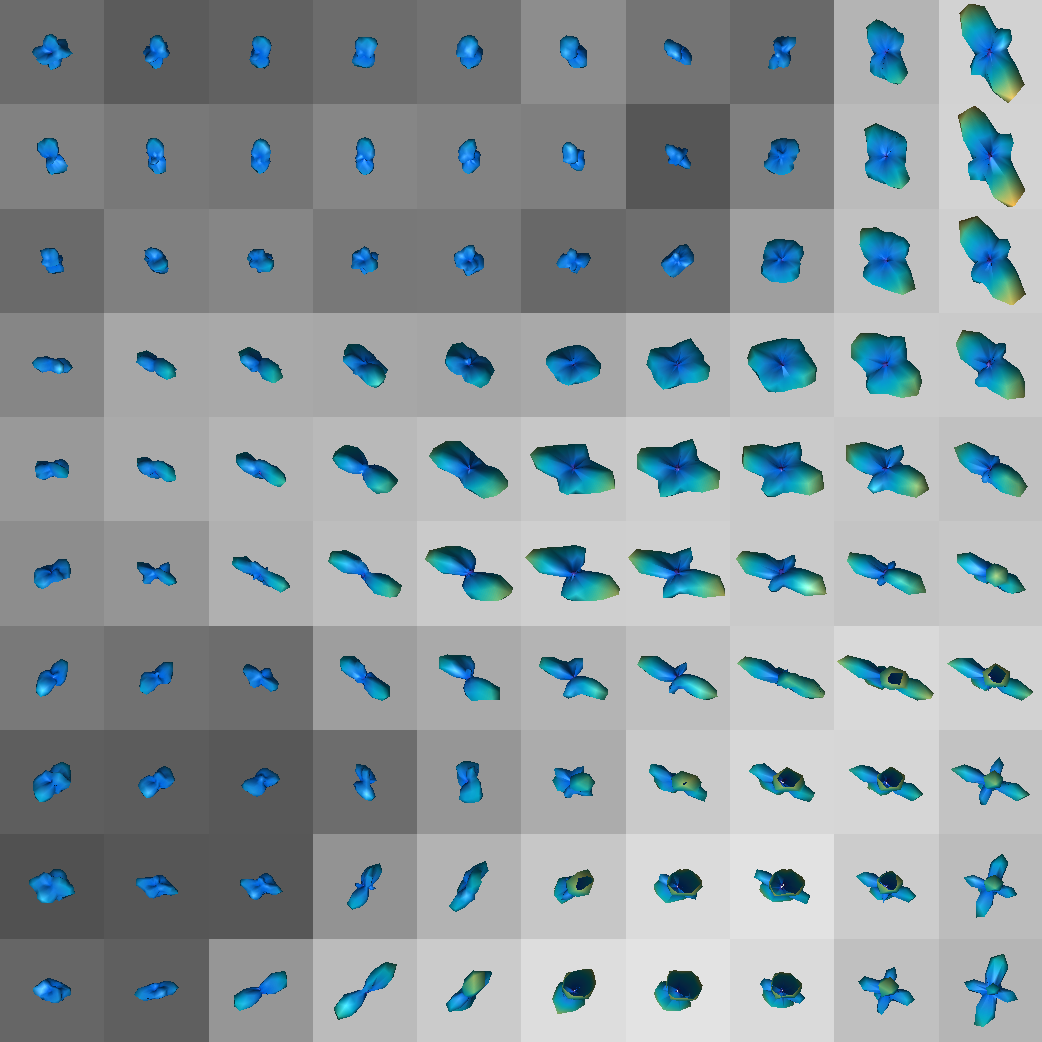}};
    \draw [thick] ([xshift=3em,yshift=7em] spatial.south west) rectangle +(5em,5em);

    \node [right=of spatial.south east,anchor=south west]
    {\includegraphics[height=8.5em]{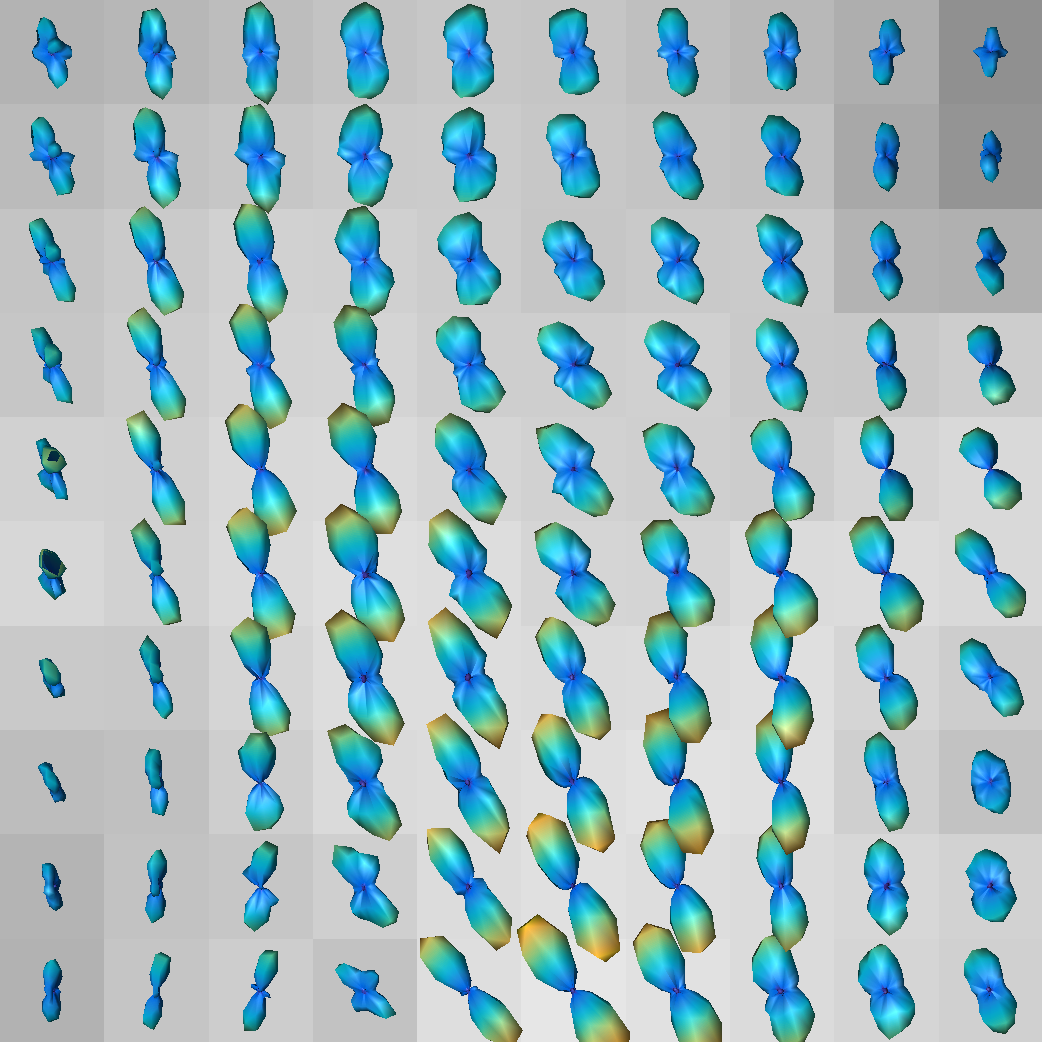}};
    \draw [thick] ([xshift=7em,yshift=0.5em] spatial.south west) rectangle +(5em,5em);

    \node [rotate=90,left=of spatial.west,anchor=south] {CSD-FC+LB};
  \end{tikzpicture}
  \caption{Comparison of CSD with \(L^2\) penalty (top) to CSD with FC and Laplace-Beltrami penalty (bottom) for \emph{in vivo} data shown in Figure~\ref{fig:hcp}.}\label{fig:hcp20}
\end{figure}

The top of figure~\ref{fig:hcp08} shows a rather complicated ``kissing'' fiber structure. As in the phantom data set above, the FC penalty tends to introduce spurious crossings in the sharply bent part, which can also be seen in the top highlighted region. The bottom highlighted region counter-intuitively shows a \emph{reduced} quality of the resolved crossing structures. We observed the same effect with spatial regularization alone, without the Laplace-Beltrami penalty. The reason for this observation is not clear.

For the slice in figure~\ref{fig:hcp20}, the highlighted region at the bottom shows a straight elongated structure, the coherence of which is significantly improved by the FC penalty. The other region shows a white matter structure extending into a gray matter area, which may or may not be an artifact due to the general tendency of the penalty to prolong elongated structures. In cases like these, validation based on other sources of knowledge about the local structures is necessary.

A notable feature of the regularized reconstructions is the significantly improved FA contrast between gray and white matter. In the figures shown, this might be attributed to the Laplace-Beltrami penalty. However, we observed the same effect with spatial regularization alone, i.e.\ the spatial penalty is able to distinguish between noisy and oriented structures using the fact that the former are not coherent with their surroundings.

\section{Discussion}

In this work, we proved the convergence of fiber continuity based spatial regularization for discrete, noisy data by showing that the natural Sobolev-type space for this regularization is compactly embedded in \(L^2\). This allows the approximation of the otherwise non-compact forward operator by finite dimensional operators.

We presented some numerical examples illustrating the performance of the method. They show a significantly more coherent ODF field compared to unregularized reconstructions, the potential to resolve crossing structures even with poor signal to noise ratio, as well as better noise suppression in isotropic areas. It is, however, not generally clear in all cases whether the observed structures are artifacts, e.g.\ spurious crossings or structures being extended into isotropic regions. Suppression or blurring of existing structures by the method may also occur. In these cases, validation of the results using other methods would be useful. In particular, tractography results from regularized reconstructions would be interesting, as they allow for comparison with existing knowledge about the structure of the brain, in particular when analyzing the widely used HCP data set.

As was already found in~\cite{reisert-kiselev}, the approach shows characteristic weaknesses in curved structures where the smoothness assumption becomes invalid. Both the theoretical results and numerical experiments suggest that the ODF reconstructions can be improved in these areas by inclusion of an additional penalty enforcing angular smoothness. The induced blurring, however, limits angular resolution, which is undesirable for crossings at acute angles. A way out of this may be locally adaptive choices of regularization parameters.

\section*{Acknowledgments}

Funding was provided by the DFG Research Training Group 1023 ``Identification in Mathematical Models''.

We would like to thank Jens Frahm and Sabine Hofer (Biomedizinische NMR Forschungs GmbH, Göttingen) for helpful discussions.

Data used in the preparation of this work were obtained from the Human Connectome Project (HCP) database. The HCP project (Principal Investigators: Bruce Rosen, M.D., Ph.D., Martinos Center at Massachusetts General Hospital; Arthur W. Toga, Ph.D., University of Southern California, Van J. Weeden, MD, Martinos Center at Massachusetts General Hospital) is supported by the National Institute of Dental and Craniofacial Research (NIDCR), the National Institute of Mental Health (NIMH) and the National Institute of Neurological Disorders and Stroke (NINDS). HCP is the result of efforts of co-investigators from the University of Southern California, Martinos Center for Biomedical Imaging at Massachusetts General Hospital (MGH), Washington University, and the University of Minnesota.

\afterpage{\clearpage}

\end{document}